\documentclass{amsart}

\newcommand{\apref}[3]{\hyperref[#2]{#1\ref*{#2}#3}}

\usepackage{enumerate}
\usepackage[latin1]{inputenc}
\usepackage{dsfont}
\usepackage{amssymb,amsthm,amsmath}

\input{xy}
\xyoption{all}

\usepackage{mathrsfs}

\theoremstyle{plain}

\newtheorem{prop}{Proposition}[section]
\newtheorem{lemma}[prop]{Lemma}

\newtheorem{thm}[prop]{Theorem}

\newtheorem{cor}[prop]{Corollary}

\theoremstyle{definition}
\newtheorem{conv}[prop]{Convention}

\newtheorem{defi}[prop]{Definition}

\newtheorem{example}[prop]{Example}

\theoremstyle{remark}
\newtheorem{remark}[prop]{Remark}

\makeatletter
\@namedef{subjclassname@2020}{%
  \textup{2020} Mathematics Subject Classification}
\makeatother

\newcommand{\TO}{\mc L}

\DeclareMathOperator{\base}{base}

\DeclareMathOperator{\ord}{ord}

\DeclareMathOperator{\tr}{tr}

\DeclareMathOperator{\Ima}{Im}
\DeclareMathOperator{\Rea}{Re}

\DeclareMathOperator{\pr}{pr}

\DeclareMathOperator{\I}{I}

\DeclareMathOperator{\Stab}{Stab}

\DeclareMathOperator{\abs}{abs}

\DeclareMathOperator{\Rel}{Rel}

\DeclareMathOperator{\CS}{CS}

\DeclareMathOperator{\Per}{Per}

\newcommand\N{\mathbb{N}}
\newcommand\Q{\mathbb{Q}}
\newcommand\R{\mathbb{R}}
\newcommand\Z{\mathbb{Z}}
\newcommand\C{\mathbb{C}}

\newcommand\T{\mathbb{T}}

\newcommand{\wh}{\widehat}

\newcommand{\eps}{\varepsilon}

\DeclareMathOperator{\id}{id}

\DeclareMathOperator{\Fct}{Fct}

\newcommand{\bmat}[4]{\begin{bmatrix} #1&#2\\#3&#4\end{bmatrix}}

\usepackage[colorlinks,breaklinks]{hyperref}
\usepackage{mathtools}
\usepackage[notref,notcite,final]{showkeys}
\usepackage[nospace,noadjust]{cite}
\usepackage{enumitem}
\usepackage{verbatim}
\usepackage{colonequals}
\usepackage{stmaryrd}
\usepackage{amssymb}
\usepackage{xcolor} 
\usepackage{color,soul}
\usepackage[scr=boondoxo]{mathalfa}
\usepackage[backgroundcolor=pink!50!white,linecolor=white,shadow]{todonotes}
\usepackage{graphicx}

\theoremstyle{definition}

\theoremstyle{plain}
\newtheorem*{namedthm}{Theorem \namedthmname}
\newcounter{namedthm}

\makeatletter
\newenvironment{named}[1]
  {\def\namedthmname{#1}%
   \refstepcounter{namedthm}%
   \namedthm\def\@currentlabel{#1}}
  {\endnamedthm}
\makeatother

\usepackage{tikz}
\usetikzlibrary{%
	math,
	automata,
	matrix,
	arrows,
	positioning,
	quotes,
	bending%
	}
\usetikzlibrary{cd, shapes, patterns, shapes.geometric, shapes.symbols, 
shapes.arrows, shapes.multipart, shapes.callouts, shapes.misc, 
decorations.pathreplacing}
\usetikzlibrary{calc}
\tikzset{
    >=stealth',
    pil/.style={
           ->,
           thick,
           shorten <=2pt,
           shorten >=2pt,}
}

\newcounter{Fig}
\setul{0.5ex}{0.3ex}
\setulcolor{red}

\DeclareRobustCommand{\SkipTocEntry}[5]{}

\makeatletter
\providecommand*\xrightarrowtriangle[2][]{%
  \ext@arrow 0055{\arrowfill@\relbar\relbar\rightarrowtriangle}{#1}{#2}}
\makeatother

\global\long\def\rwund{\textcolor{red}{\wund}}

\begin{document}
\global\long\def\defset#1#2{\left\{  #1\;\middle|\;#2\right\}  }
\global\long\def\dif#1{\operatorname{d}\!#1}
\global\long\def\i{\mathrm{i}}
\global\long\def\H{\mathbb{H}}
\global\long\def\R{\mathbb{R}}
\global\long\def\Z{\mathbb{Z}}
\global\long\def\Q{\mathbb{Q}}
\global\long\def\C{\mathbb{C}}
\global\long\def\N{\mathbb{N}}
\global\long\def\PSLR{\mathrm{PSL}_2(\R)}
\global\long\def\x{\mathrm{x}}
\global\long\def\hyp{\mathrm{h}}
\global\long\def\id{\operatorname{id}}
\global\long\def\tr#1{\operatorname{tr}(#1)}
\global\long\def\TO#1{\mathcal{L}_{#1}}
\global\long\def\fTO#1{\widetilde{\mathcal{L}}_{#1}}
\global\long\def\Return{\mathscr{R}}
\global\long\def\act{\boldsymbol{.}}
\global\long\def\CrSc{\widehat{\operatorname{C}}}
\global\long\def\BrU{\mathrm{C}}
\global\long\def\BrS{\mathcal{C}}
\global\long\def\Cs#1{\operatorname{C}_{#1}}
\global\long\def\Csred#1{\operatorname{C}_{#1,\mathrm{red}}}
\global\long\def\Csacc#1{\operatorname{C}_{#1,\mathrm{acc}}}
\global\long\def\CS#1{\widehat{\operatorname{C}_{#1}}}
\global\long\def\Iset#1{I_{#1,\st}}
\global\long\def\Jset#1{J_{#1,\st}}
\global\long\def\Plussp#1{\mathrm{H}_+(#1)}
\global\long\def\Minussp#1{\mathrm{H}_-(#1)}
\global\long\def\PMsp#1{\mathrm{H}_\pm(#1)}
\global\long\def\MPsp#1{\mathrm{H}_\mp(#1)}
\global\long\def\Iacc#1{I_{#1,\mathrm{acc}}}
\global\long\def\Jacc#1{J_{#1,\mathrm{acc}}}
\global\long\def\base#1{\operatorname{bp}(#1)}
\global\long\def\fund{\mathcal{F}}
\global\long\def\Iso#1{\mathrm{ISO}(#1)}
\global\long\def\iso#1{\mathrm{I}(#1)}
\global\long\def\Rel{\mathrm{REL}}
\global\long\def\intiso#1{\mathrm{int\,}\iso{#1}}
\global\long\def\extiso#1{\mathrm{ext\,}\iso{#1}}
\global\long\def\mittelp#1{\operatorname{c}(#1)}
\global\long\def\radius#1{\operatorname{r}(#1)}
\global\long\def\summit#1{\operatorname{s}(#1)}
\global\long\def\UTB#1{\operatorname{S}\!#1}
\global\long\def\quod#1#2{#1\diagdown#2}
\global\long\def\Orbi{\mathbb{X}}
\global\long\def\GeoFlow{\widehat{\Phi}}
\global\long\def\cyc#1#2{\operatorname{cyc}_{#1}(#2)}
\global\long\def\cycstar#1#2{\operatorname{cyc}^*_{#1}(#2)}
\global\long\def\cycset#1#2{\operatorname{Cyc}_{#1,#2}}
\global\long\def\cycstarset#1#2{\operatorname{Cyc}^*_{#1,#2}}
\global\long\def\cycnext#1#2{g_{#2}(#1)}
\global\long\def\cycstarnext#1#2{g^*_{#2}(#1)}
\global\long\def\cyctrans#1#2{u_{#1,#2}}
\global\long\def\point#1#2{#2_{#1}}
\global\long\def\Per{\operatorname{Per}}
\global\long\def\Sub{\operatorname{Sub}}
\global\long\def\st{\operatorname{st}}
\global\long\def\T#1{\mathscr{T}(#1)}
\global\long\def\abs#1{\left|#1\right|}
\global\long\def\geo{\mathscr{g}}
\global\long\def\Geo{\mathscr{G}}
\global\long\def\Ball#1#2{\mathrm{B}_{#1}(#2)}
\global\long\def\Vanish{\mathrm{Van}}
\global\long\def\retime#1{t^+_{#1}}
\global\long\def\pretime#1{t^-_{#1}}
\global\long\def\ittime#1{\mathrm{t}_{#1}}
\global\long\def\ittrans#1{\mathrm{g}_{#1}}
\global\long\def\itindex#1{\mathrm{k}_{#1}}
\global\long\def\Trans#1#2#3{\mathcal{G}_{#1}(#2,#3)}
\global\long\def\Transprime#1#2#3{\mathcal{G}'_{#1}(#2,#3)}
\global\long\def\Past#1#2#3{\mathcal{V}_{#1}(#2,#3)}
\global\long\def\Heir#1#2{H_{#1}(#2)}
\global\long\def\acc{\mathrm{acc}}
\global\long\def\eX{\mathrm{X}}
\global\long\def\eY{\mathrm{Y}}
\global\long\def\eR{\mathrm{R}}
\global\long\def\eZ{\mathrm{Z}}
\global\long\def\Removec#1#2{L_{#1}(#2)}
\global\long\def\Index{\widetilde{A}}
\global\long\def\gbound#1{\geo#1}
\global\long\def\ReturnGraph#1{\mathrm{RG}_{#1}}
\global\long\def\Cir#1{\mathrm{Cir}_{#1}}
\global\long\def\ram#1{\operatorname{ram}(#1)}
\global\long\def\Ram#1{\mathrm{Ram}_{#1}}
\global\long\def\fixp#1#2{\operatorname{f}_{#1}(#2)}
\global\long\def\Stab#1#2{\operatorname{Stab}_{#1}(#2)}
\global\long\def\standh#1{\mathrm{h}_{#1}}
\global\long\def\countit#1{\varphi(#1)}
\global\long\def\Unit{\operatorname{U}}
\global\long\def\ct{\operatorname{ct}}
\global\long\def\Att#1#2{\operatorname{Att}_{#1}(#2)}
\global\long\def\nbAtt#1{\operatorname{Att}_{#1}}
\global\long\def\IAtt#1#2{I(\Att{#1}{#2})}
\global\long\def\conv#1#2{\operatorname{conv}_{#1}(#2)}
\global\long\def\slow{\mathrm{slow}}
\global\long\def\fast{\mathrm{fast}}
\global\long\def\Fct{\operatorname{Fct}}
\global\long\def\edge#1{\xrightarrowtriangle{#1}}

\global\long\def\too{\longrightarrow}
\global\long\def\mapstoo{\longmapsto}
\global\long\def\iff{\Longleftrightarrow}
\global\long\def\kund{\mathcal{K}}
\global\long\def\wund{\mathcal{W}}
\global\long\def\REL#1{\mathrm{REL}(#1)}
\global\long\def\RELL{\mathrm{REL}}
\global\long\def\P{\mathrm{P}}
\global\long\def\Plusspp#1{\mathrm{H}_+^{\P}(#1)}
\global\long\def\Minusspp#1{\mathrm{H}_-^{\P}(#1)}
\global\long\def\pr{\operatorname{pr}_{\infty}}
\global\long\def\prim{\operatorname{Prim}}
\global\long\def\underiso#1{\mathscr{W}(#1)}
\global\long\def\neindex{A'}
\global\long\def\J{\mathrm{J}}
\global\long\def\summit#1{\operatorname{s}(#1)}

\title[Strict Transfer Operators for Non-Compact Orbisurfaces]{Strict Transfer Operator Approaches for Non-Compact Hyperbolic Orbisurfaces}
\author[P.\@ Wabnitz]{Paul Wabnitz}
\address{Paul Wabnitz, University of Bremen, Department~3 - Mathematics, Institute for Dynamical Systems, Bibliothekstr.~5, 28359~Bremen, Germany}
\email{pwabnitz@uni-bremen.de}

\begin{abstract}
By building on former results in~\cite{Pohl_Wab} and the cusp expansion algorithm developed by Pohl~\cite{Pohl_diss}, we construct strict transfer operator approaches in the sense of Fedosova and Pohl~\cite{FP_NECM} for geometrically finite developable hyperbolic orbisurfaces of infinite area without cusps. Together with the cusp expansion algorithm for orbisurfaces with cusps, this provides strict transfer operator approaches for all hyperbolic orbifolds fulfilling mild assumptions.
For every such orbisurface we obtain explicit transfer operator families for which, by virtue of~\cite{FP_NECM}, the Fredholm determinant function is seen to be identical to the associated Selberg zeta function.
\end{abstract}

\maketitle
\tableofcontents

\section*{Introduction}\label{SEC:intro}

Consider a geometrically finite non-cocompact Fuchsian group~$\Gamma$ (possibly including elliptic elements) acting on the hyperbolic plane~$\H$.
In~\cite{Pohl_Wab} we established the concept of \emph{sets of branches} for the geodesic flow on the orbit space~$\Orbi=\quod{\Gamma}{\H}$.
These sets of branches are, in short, collections of well-chosen subsets of the unit tangent bundle~$\UTB\H$ whose union represents a cross section\footnote{Our notion of cross sections differs from the usual notion of Poincar\'e cross sections. We refer to~\cite[Section~2]{Pohl_Wab} for the details.} for the geodesic flow on said orbit space (a comprehensive definition is provided in Section~\ref{SEC:notdef} below).
The main result of~\cite{Pohl_Wab} showed that cross sections obtained in this way provide all the structure necessary for the implementation of a \emph{strict transfer operator approach} as developed by Fedosova and Pohl in~\cite{FP_NECM}.
This approach yields a symbolic dynamics semi-conjugate to the first return map for the discretization of the geodesic flow induced by that cross section.
The one-parameter family~$\{\fTO{s}\}_{s\in\C}$ of transfer operators arising from this dynamics is, by virtue of~\cite[Theorem~4.2]{FP_NECM}, sufficiently well-behaved (namely nuclear of order~$0$ on a certain Banach space of holomorphic functions) for their Fredholm determinants to exist and constitute a representation of the Selberg zeta function on~$\Orbi$.
This, together with the meromorphic continuability of the parameterization~$s\mapsto\fTO{s}$ as well as of the zeta function itself, makes it possible to effectively study resonances and resonant states of the Laplace--Beltrami operator on~$\Orbi$ in terms of~$1$-eigenspaces of those transfer operators.
We refer the reader to~\cite[Sections~4.2--4.4]{FP_NECM} (or alternatively to~\cite[Sections~3.7--3.8]{Pohl_Wab}) for complete definitions and purposes of the strict transfer operator approach, the \emph{structure tuple} it requires, and the fast transfer operator family arising from it.

In this paper we provide explicit sets of branches for every non-compact developable hyperbolic orbisurfaces fulfilling a mild technical condition (see Condition~\eqref{condA} below).
This finishes the proof of the following result, thereby justifying the efforts undertaken in~\cite{Pohl_Wab}:

\begin{named}{A}\label{THMA}
Let~$\Gamma$ be an admissible group of isometries of the hyperbolic plane.
Then~$\Gamma$ admits a strict transfer operator approach with structure tuple given explicitly.
\end{named}

Here, by an \emph{admissible} group of isometries we mean a geometrically finite non-cocompact Fuchsian group with hyperbolic elements that fulfills Condition~\eqref{condA}.

By virtue of previous work by various researches, Theorem~\ref{THMA} is known to hold in many situations.
We refer the reader to the introduction of~\cite{Pohl_Wab} for an extensive list of references.
Of utmost significance for our approach is the so called \emph{cusp expansion algorithm} developed by Pohl in~\cite{Pohl_diss}, which we review in Section~\ref{SEC:cuspexp} and which is known to give rise to sets of branches for admissible groups engendering cusps in their orbit spaces.
In Section~\ref{SEC:cuspless} we then consider admissible groups of infinite covolume without parabolic elements, for which, so far and in this generality, no systematic approach is known.
We provide one by an auxiliary group argument that allows us to apply the cusp expansion algorithm ``out of context''.
More precisely, we construct an auxiliary Fuchsian group with exactly one conjugacy class of parabolic elements, to which we can apply the cusp expansion algorithm.
The set of branches we obtain from that is then seen to induce a set of branches for the initial group as well.
In both cases, all constructions are completely algorithmic and therefore provide explicit transfer operator families for every admissible group of isometries of the hyperbolic plane.

Sections~\ref{SEC:notdef} and~\ref{SEC:isoms} are of preliminary nature and provide the necessary definitions as well as technical observations which we require for our proofs.

\newpage
\addtocontents{toc}{\SkipTocEntry}
\subsection*{Acknowledgement} 
This research was funded by the Deutsche Forschungsgemeinschaft (DFG, German 
Research Foundation) -- project no.~264148330 and PO~1483/2-1. The author 
wishes to thank the Hausdorff Institute for Mathematics in Bonn for excellent working conditions during the HIM trimester program ``Dynamics: Topology and Numbers,'' where part of this manuscript was prepared.
\vspace{3em}

\section{Notations and definitions}\label{SEC:notdef}

\subsection*{Isometries of the hyperbolic plane}\label{SUBSEC:PrelimIsom}

Throughout this note we make exclusive use of the upper half-plane model
\[
\H=\defset{z\in\C}{\Ima{z}>0}\,,\qquad\dif{s}^2=(\Ima{z})^{-2}\dif{z}\dif{\overline{z}}\,.
\]
The geodesic boundary~$\partial_{\geo}\H$ of~$\H$ can be identified with the Alexandroff extension~$\widehat\R\coloneqq\R\cup\{\infty\}$ of the real line equipped with Euclidean topology.
For a subset~$K$ of~$\overline{\H}^{\geo}\coloneqq\H\cup\partial_{\geo}\H$ we write~$\partial K$ for its boundary in~$\H$,~$\partial_{\geo}K$ for its boundary in~$\overline{\H}^{\geo}$, and~$\geo K$ for its geodesic boundary.
Its closure in~$\H$ is denoted by~$\overline{K}$ and its closure in~$\overline{\H}^{\geo}$ by~$\overline{K}^{\geo}$.

The group of orientation preserving isometries on~$\H$ is isomorphic to
\[
\PSLR=\mathrm{SL}_2(\R)\slash\{\pm\id\}\,,
\]
whose elements we consider to act (from the left) on~$\H$ by linear fractional transformations. These transformation extend smoothly to~$\overline{\H}^{\geo}$. We write~$g\act z$ for the action of~$g\in\PSLR$ on~$z\in\overline{\H}^{\geo}$.

A \emph{Fuchsian group}~$\Gamma$ is a discrete subgroup of~$\PSLR$. We denote its orbit space by
\[
\Orbi=\quod{\Gamma}{\H}
\]
and the associated canonical quotient map by
\[
\pi\colon\H\too\Orbi\,.
\]
Since~$\Gamma$ acts properly discontinuously on~$\H$, its orbit space inherits via~$\pi$ the structure of a two-dimensional developable Riemannian orbifold, called a developable orbisurface.

A Fuchsian group is called \emph{geometrically finite} if it has a fundamental domain that is a geometrically finite, exact, convex polyhedron (see, e.g., \cite{Ratcliffe}).
Given the presence of periodic geodesics (see below), geometric finiteness allows for only two types of hyperbolic ends of~$\Orbi$, namely cusps and funnels, and~$\Orbi$ possesses not more than finitely many of each.
Further,~$\Orbi$ might have conical singularities, which is the case if and only if~$\Gamma$ contains elliptic elements.
A geometrically finite Fuchsian group~$\Gamma$ is called \emph{cocompact} if~$\Orbi$ is compact.
We call~$\Gamma$ \emph{non-cocompact}, if this is not the case, meaning that~$\Orbi$ has hyperbolic ends.
We emphasize that this should not be confused with convex cocompactness, which demands the convex core of~$\Orbi$ to be compact and~$\Orbi$ itself to be a proper surface, i.e., void of conical singularities.
Hence, a Fuchsian group can be non-cocompact and convex cocompact at the same time.
We avoid impending confusion in that area by instead speaking of (developable) orbisurfaces with hyperbolic ends.

We denote the limit set of~$\Gamma$ by~$\Lambda(\Gamma)$ and the ordinary set by~$\Omega(\Gamma)$.
The limit set can be finite, in which case~$\Gamma$ is called~\emph{elementary}.
We refer to~\cite{Katok_fuchsian} for a complete characterization of elementary Fuchsian groups.
If~$\Lambda(\Gamma)$ contains more than two distinct points, then it contains infinitely many and is a perfect set in~$\wh\R$.

Define
\[
\widehat\R_{\st}\coloneqq\Lambda(\Gamma)\setminus\bigcup\{\text{parabolic fixed points of }\Gamma\}\,.
\]
Then~$\widehat\R_{\st}$ contains, in particular, all hyperbolic fixed points of~$\Gamma$, and~$\Gamma$ is cocompact if and only if~$\widehat\R=\widehat\R_{\st}$.
For every subset~$M$ of~$\widehat\R$ we set
\[
M_{\st}\coloneqq M\cap\widehat\R_{\st}\,.
\]

\subsection*{Geodesics}\label{SUBSEC:PrelimGeod}

Let~$\gamma\in\mathcal{C}^\infty(\R;\H)$.
Then~$\gamma$ is a geodesic on~$\H$ if and only if~$\gamma$ is injective and~$\gamma(\R)$ is a (generalized) semicircle perpendicular to~$\partial_{\geo}\H$.
If this is the case then~$\gamma(\R)$ is called a \emph{complete geodesic segment} and we write
\[
(\gamma(-\infty),\gamma(+\infty))_{\H}\coloneqq\gamma(\R)\,.
\]
Likewise, for~$z_1,z_2\in\overline{\H}^{\geo}$ we denote the geodesic arc joining~$z_1$ to~$z_2$ by~$(z_1,z_2)_{\H}$ and call it a \emph{geodesic segment}.
Hence, the geodesic segment~$(z_1,z_2)_{\H}$ is complete if and only if~$\{z_1,z_2\}\subseteq\partial_{\geo}\H$. 
We consider all geodesics to be of unit speed, i.e.,~$\abs{\gamma'(t)}=1$ for all~$t\in\R$.
In order to prevent confusion with intervals in~$\R$ we will denote those by~$(z_1,z_2)_\R$.
The notation~$(z_1,z_2)$ is thus reserved exclusively for tuples.
We utilize the same notation for closed and half-open intervals, or for geodesic segments containing either one or both of their respective endpoints (including those with one or both endpoints in~$\partial_\geo\H$).

A complete geodesic segment is called \emph{vertical} if either one of its endpoints equals~$\infty$, and \emph{non-vertical} otherwise.
A geodesic segment is called \emph{vertical} if it is contained in a vertical complete geodesic segment, and \emph{non-vertical} if this is not the case.

For every complete geodesic segment the associated parameterization can be chosen to be of unit speed and we will always assume that this is the case.
Two (unit speed) geodesics~$\gamma_1$ and~$\gamma_2$ are considered equivalent if
\[
(\gamma_1(-\infty),\gamma_1(+\infty))=(\gamma_2(-\infty),\gamma_2(+\infty))\,,
\]
and we denote the set of equivalence classes of geodesics on~$\H$ by~$\Geo(\H)$.

Let
\[
\pr\colon\H\too\partial_{\geo}\H
\]
be the geodesic projection from~$\infty$ onto the boundary of~$\H$.
Note that, in contrast to~\cite{Pohl_diss}, we do not extend this projection to~$\overline{\H}^{\geo}$ here, meaning that for~$x\in\R$ we get
\[
\pr^{-1}(x)=(x,\infty)_{\H}\subseteq\H\,.
\]

Let~$\Gamma$ be a geometrically finite Fuchsian group with orbit space~$\Orbi$.
The geodesics on~$\Orbi$ can be identified with the concatenations
\[
\widehat{\gamma}\coloneqq\pi\circ\gamma
\]
for~$\gamma\in\Geo(\H$).
We denote by~$\Geo(\Orbi)$ the set of all (equivalence classes of) geodesics on~$\Orbi$ and by~$\Geo_{\Per}(\Orbi)$ the subset of all (equivalence classes of) periodic geodesics on~$\Orbi$.
For the sake of abbreviation we also write
\[
\Geo_{\Per,\Gamma}(\H)\coloneqq\defset{\gamma\in\Geo(\H)}{\pi(\gamma)\in\Geo_{\Per}(\Orbi)}\,.
\]

For~$g\in\Gamma$ elliptic or parabolic we denote its unique fixed point in~$\overline{\H}^\geo$ by~$\fixp{}{g}$.
If~$g\in\Gamma$ is hyperbolic then its attracting fixed point is denoted by~$\fixp{+}{g}$ and its repelling one by~$\fixp{-}{g}$.
The set
\begin{equation}\label{EQNDEF:EX}
\begin{aligned}
E(\Orbi)\coloneqq&\defset{(\gamma(+\infty),\gamma(-\infty))}{\gamma\in\Geo_{\Per,\Gamma}(\H)}\\
=&\defset{(\fixp{+}{h},\fixp{-}{h})}{h\in\Gamma\text{ hyperbolic}}
\end{aligned}
\end{equation}
will play a central role for our approach.
Recall that~$E(\Orbi)$ lies dense in~$\Lambda(\Gamma)\times\Lambda(\Gamma)$ whenever the Fuchsian group~$\Gamma$ contains hyperbolic elements (see~\cite{Eberlein_curvedMflds}).
The equivalence class~$\alpha(h)$ of geodesics joining~$\fixp{-}{h}$ to~$\fixp{+}{h}$ is called the~\emph{(hyperbolic) axis} of~$h$, and the equality of the two sets in~\eqref{EQNDEF:EX} is due to the bijection between conjugacy classes of hyperbolic elements in~$\Gamma$ and periodic geodesics on~$\Orbi$~\cite[Theorem~3.30]{Eberlein_curvedMflds}.

We write~$\UTB\H$ for the unit tangent bundle of~$\H$.
Each~$\nu=(z,\vec{\nu})\in\UTB\H$ uniquely determines a geodesic~$\gamma_\nu\in\Geo(\H)$ via the rule
\[
\gamma_\nu(0)=z\,,\qquad\gamma_\nu^{\prime}(0)=\vec{\nu}\,.
\]
We further denote by
\[
\mathrm{bp}\colon\UTB\H\ni(z,\vec{\nu})\mapstoo z\in\H
\]
the projection onto base points.

\subsection*{Sets of branches}\label{SUBSEC:PrelimSoB}

We now recall the notion of a set of branches from~\cite{Pohl_Wab}, which will be the central object for the considerations that follow.
A detailed discussion of the defining properties and their consequences can be found ibid.
We assume~$\Gamma$ to be a geometrically finite, non-cocompact Fuchsian group containing hyperbolic elements and denote its orbit space by~$\Orbi$.

\begin{defi}\label{DEF:setofbranches}
Let~$N\in\N$ and let~$\Cs1,\ldots, \Cs N$ be subsets of~$\UTB\H$. Set~$A\coloneqq\{1,\dots,N\}$, 
\[
\BrS \coloneqq \defset{\Cs j }{ j\in A }\,\qquad\text{and}\qquad\BrU \coloneqq \bigcup_{j\in A} \Cs j.
\]
We call $\BrS$ a \emph{set of branches for the geodesic flow 
on}~$\Orbi$ if it satisfies the following properties:

\begin{enumerate}[label=$\mathrm{(B\arabic*)}$, ref=$\mathrm{B\arabic*}$, itemsep=2ex]
\item\label{BP:closedgeodesicsHtoX}
For each~$j\in A$ there exists~$\nu\in\Cs{j}$ such that $\pi(\gamma_\nu)$ is a 
periodic geodesic on~$\Orbi$.
\item\label{BP:completegeodesics}
For each~$j\in A$, the set~$\overline{\base{\Cs{j}}}$ is a complete geodesic 
segment in~$\H$ and its endpoints are in~$\widehat\R\setminus\widehat\R_{\st}$. In particular, for each $j\in A$, the 
set~$\H\setminus\overline{\base{\Cs{j}}}$ decomposes uniquely into two (geodesically) convex open half-spaces.
\item\label{BP:pointintohalfspaces}
For each~$j\in A$, all elements of~$\Cs{j}$ point into the same open half-space 
relative to~$\overline{\base{\Cs{j}}}$. We denote this half-space by~$\Plussp{j}$ and set
\[
\Minussp{j}\coloneqq\H\setminus\bigl(\overline{\base{\Cs{j}}}\cup\Plussp{j}\bigr)\,.
\]
Further, we denote by~$I_j$ the largest open subset of~$\wh\R$ that is contained in~$\geo\Plussp{j}$, and by~$J_j$ the largest open subset of~$\wh\R$ contained in~$\geo\Minussp{j}$.
\item\label{BP:coverlimitset}
The~$\Gamma$-translates of~$\defset{I_j}{j\in A}$ cover the set~$\wh\R_{\st}$, i.e.,
\[
\wh\R_{\st}\subseteq\bigcup_{j\in A}\bigcup_{g\in\Gamma}g\act I_j\,.
\]
\item\label{BP:allvectors}
For each $j\in A$ and each pair~$(x,y)\in \Iset{j}\times\Jset{j}$ there exists a (unique) vector~$\nu\in\Cs{j}$ such that
\[
(x,y)=(\gamma_\nu(+\infty),\gamma_\nu(-\infty))\,.
\]
\item\label{BP:disjointunion} If $\overline{\base{\Cs{j}}}\cap 
g\act\overline{\base{\Cs{k}}}\ne\varnothing$ for some~$j,k\in A$ and~$g\in\Gamma$, then either $j=k$ and $g=\id$, or $\PMsp{j} = g\act\MPsp{k}$.
\item\label{BP:intervaldecomp}  For each pair~$(a,b)\in A\times A$ there exists a (possibly empty) subset~$\Trans{}{a}{b}$ of~$\Gamma$ such that
\begin{enumerate}[label=$\mathrm{(\alph*)}$, ref=$\mathrm{\alph*}$]

\item\label{BP:intervaldecompGdecomp}
 for all~$j\in A$ we have 
\begin{align*}
\bigcup_{k\in A}\bigcup_{g\in\Trans{}{j}{k}}g\act I_{k}&\subseteq 
I_j
\intertext{and}
\bigcup_{k\in 
A}\bigcup_{g\in\Trans{}{j}{k}}g\act\Iset{k}&=\Iset{j}\,,
\end{align*}
and these unions are disjoint,

\item\label{BP:intervaldecompGgeod}
for each pair~$(j,k)\in A\times A$, each $g\in\Trans{}{j}{k}$ and each pair of points~$(z,w)\in\base{\Cs{j}}\times g\act\base{\Cs{k}}$, the geodesic segment~
$(z,w)_{\H}$ is non-empty, is contained in~$\Plussp{j}$ and does not intersect
$\Gamma\act\BrU$,

\item\label{BP:intervaldecompback}
for all~$j\in A$ we have 
\[
 \Jset{j}\subseteq\bigcup_{k\in A}\bigcup_{h\in\Trans{}{k}{j}}h^{-1}\act\Jset{k}\,.
\]
\end{enumerate}
\end{enumerate}

We call the sets~$\Cs{j}$, $j\in A$, the \emph{branches} of~$\BrS$, and $\BrU$ the \emph{branch union}. Further, we call the sets~$\Trans{}{j}{k}$, $j,k\in A$, the \emph{(forward) transition sets} of~$\BrS$, with $\Trans{}{j}{k}$ being the \emph{(forward) transition set from~$\Cs{j}$ to~$\Cs{k}$}.
\end{defi}

\begin{remark}\label{REM:noadmandnoncollap}
By comparing Definition~\ref{DEF:setofbranches} to~\cite[Definition~3.1]{Pohl_Wab} one notices that the latter one includes two additional notions, namely those of \emph{admissible} and \emph{non-collapsing} sets of branches.
Since, for the duration of this paper, we are only interested in the mere existence of sets of branches for a given Fuchsian group, we do not have to take those notions into account here..
This is due to~\cite[Proposition~5.9]{Pohl_Wab} where it has been shown that every set of branches without those properties can be transformed into one that does have them.
\end{remark}

We further recall the following consequential property.

\begin{lemma}[Proposition~3.7 in \cite{Pohl_Wab}]\label{LEM:allperiodintersect}
Every~$\gamma\in\Geo_{\Per,\Gamma}(\H)$ intersects~$\Gamma\act\BrU$.
\end{lemma}

In~\cite[Proposition~3.24]{Pohl_Wab} it has been shown that, for a set of branches~$\BrS$, the associated branch union~$\BrU$ represents a cross section for the geodesic flow on~$\Orbi$.
A strong cross section can be obtained by considering~$\pi(\BrU_{\st})$ instead, where
\[
\BrU_{\st}\coloneqq\bigcup_{j\in A}\Cs{j,\st}\qquad\text{for}\qquad\Cs{j,\st}\coloneqq\defset{\nu\in\Cs j}{(\gamma_{\nu}(+\infty),\gamma_{\nu}(-\infty))\in\widehat\R_{\st}\times\widehat\R_{\st}}\,.
\]
The main result of~\cite{Pohl_Wab} now reads as follows.

\begin{thm}\label{THM:mainthmPW}
Let~$\Gamma$ be a geometrically finite Fuchsian group for which there exists a set of branches. Then~$\Gamma$ admits a strict transfer operator approach.
\end{thm}

The definition of strict transfer operator approaches has been included in~\cite[Section~3.7]{Pohl_Wab}.
We also refer the reader to~\cite{FP_NECM}, in particular to Theorem~4.2 ibid., which states the relation to transfer operators yielding a representation of the Selberg zeta function as mentioned in the introduction.

\section{Isometric spheres}\label{SEC:isoms}

In this section we collect the technical background information as well as some observations about isometric spheres in our setting.

Let~$\Gamma$ be geometrically finite Fuchsian group.
By conjugation in~$\PSLR$ we can always achieve that~$\infty\in\wh\R\setminus\wh\R_{\st}$, which we assume for simplicity.
Then the stabilizer subgroup
\[
\Gamma_{\infty}\coloneqq\Stab{\Gamma}{\infty}
\]
is either trivial, if~$\infty\in\Omega(\Gamma)$, or generated by a single parabolic transformation
\[
t_{\lambda}\coloneqq\bmat{1}{\lambda}{0}{1}\,,
\]
where~$\lambda>0$ is uniquely determined.
In the latter case~$\infty$ represents a cusp of~$\Orbi$ and~$\lambda$ is called the \emph{cusp width} of~$\pi(\infty)$.
To each~$g=\begin{bsmallmatrix}a&b\\c&d\end{bsmallmatrix}\in\Gamma\setminus\Gamma_\infty$ we assign its \emph{isometric sphere}
\[
\iso{g}\coloneqq\defset{z\in\H}{\abs{g'(z)}=1}=\defset{z\in\H}{\abs{z+\tfrac{d}{c}}=\tfrac{1}{\abs{c}}}\,.
\]
We denote by~$\Iso{\Gamma}$ the set of all isometric spheres.
For~$\I\in\Iso{\Gamma}$ we call~$g\in\Gamma\setminus\Gamma_\infty$ a \emph{generator} of~$\I$ whenever~$\I=\iso{g}$.
We further define the \emph{interior} and the \emph{exterior},
\begin{align*}
\mathrm{int}\,\I&=\intiso{g}\coloneqq\defset{z\in\H}{\abs{g'(z)}>1}
\intertext{and}
\mathrm{ext}\,\I&=\extiso{g}\coloneqq\defset{z\in\H}{\abs{g'(z)}<1}
\end{align*}
of~$\I=\iso{g}\in\Iso{\Gamma}$.
We recall the following properties of isometric spheres (where again~$g=\begin{bsmallmatrix}a&b\\c&d\end{bsmallmatrix}\in\Gamma\setminus\Gamma_{\infty}$):
\begin{enumerate}[label=$(\mathrm{ISO}\arabic*)$,ref=$\mathrm{ISO}\arabic*$]
\item\label{ITEM:iso:eqn}
The equations
\begin{align*}
g\act\iso{g}&=\iso{g^{-1}}\,,\\
g\act\intiso{g}&=\extiso{g^{-1}}\,,
\intertext{and}
g\act\extiso{g}&=\intiso{g^{-1}}
\end{align*}
hold true.

\item\label{ITEM:iso:height}
For every~$z\in\iso{g}$ we have~$\Ima(g\act z)=\Ima{z}$.

\item\label{ITEM:iso:centers}
The center of~$\iso{g}$ is given by~$g^{-1}\act\infty$ and the center of~$\iso{g^{-1}}$ by~$g\act\infty$.

\item\label{ITEM:iso:radii}
The radii of~$\iso{g}$ and~$\iso{g^{-1}}$ coincide and equal~$\abs{c}^{-1}$.

\item\label{ITEM:iso:gener}
We have~$\iso{g_1}=\iso{g_2}$ for~$g_1,g_2\in\Gamma\setminus\Gamma_{\infty}$ if and only if~$g_1g_2^{-1}\in\Gamma_\infty$, meaning that generators of isometric spheres are uniquely determined up to left multiplication with elements in~$\Gamma_{\infty}$.
Furthermore, for~$g\in\Gamma\setminus\Gamma_{\infty}$ and~$n\in\Z$ we have~$\iso{gt_{\lambda}^n}=t_\lambda^{-n}\act\iso{g}$. (\cite[Lemmas~6.1.2 and~6.1.3]{Pohl_diss}).

\item\label{ITEM:iso:boundradii}
The radii of isometric spheres are bounded from above.

\item\label{ITEM:iso:locfin}
The sets~$\Iso{\Gamma}$ and~$\defset{\intiso{g}}{g\in\Gamma\setminus\Gamma_\infty}$ are locally finite respectively. (See \cite[Proposition~6.1.5]{Pohl_diss} for the case that~$\Gamma$ has cusps, and~\cite[Lemma~4.3.7]{Pohl_isofunddom} for the case that it does not.)
\end{enumerate}

Statement~\eqref{ITEM:iso:gener} implies that the map
\[
\Gamma\ni\bmat{a}{b}{c}{d}\mapstoo\abs{c}\in\R
\]
is constant on the double coset~$\Gamma_\infty g\Gamma_\infty$, for every~$g\in\Gamma$.
On a related note, the following result shows that with the center of an isometric spheres its radius is already uniquely given as well.

\begin{lemma}\label{LEM:isoconcentric}
Let~$\I,\J\in\Iso{\Gamma}$ be concentric.
Then~$\I=\J$.
\end{lemma}

\begin{proof}
Let~$g=\begin{bsmallmatrix}a&b\\c&d\end{bsmallmatrix}\in\Gamma\setminus\Gamma_\infty$ be a generator of~$\I$ with~$c>0$.
By~\eqref{ITEM:iso:centers} resp.~\eqref{ITEM:iso:radii} the center of~$\I$ is then given by~$-\tfrac{d}{c}$ and its radius by~$\tfrac{1}{c}$.
By assumption, a generator of~$\J$ must be of the form~$h\coloneqq\begin{bsmallmatrix}x&y\\rc&rd\end{bsmallmatrix}$ with some~$r\in\R\setminus\{0\}$.
From its determinant condition we obtain~$dx-cy=\tfrac{1}{r}$.
Then
\[
gh^{-1}=\begin{bmatrix}a&b\\c&d\end{bmatrix}\begin{bmatrix}rd&-y\\-rc&x\end{bmatrix}=\begin{bmatrix}adr-bcr&bx-ay\\0&dx-cy\end{bmatrix}=\begin{bmatrix}r&bx-ay\\0&\tfrac{1}{r}\end{bmatrix}\,.
\]
Hence,~$gh^{-1}\in\Gamma_\infty$.
Since~$\Gamma_\infty$ is either trivial or cyclic with generator~$\begin{bsmallmatrix}1&\lambda\\0&1\end{bsmallmatrix}$ for some~$\lambda\in\R$, this implies~$r=1$, which in turn yields the assertion.
\end{proof}

Later on we will also require the following observation, which is concerned with the interrelation between the isometric spheres an the limit set of~$\Gamma$.

\begin{lemma}\label{LEM:isolimit}
Let~$\Gamma$ be non-conjugate to any group generated by~$z\mapsto-\tfrac{1}{z}$ and~$z\mapsto\lambda z$ for~$\lambda>1$ and suppose that~$\Lambda(\Gamma)$ consists of more than one point.
For~$g\in\Gamma\setminus\Gamma_\infty$ we set
\[
\underiso{g}\coloneqq\left(\pr(\iso{g})\right)^{\circ}=\left(\geo\intiso{g}\right)^{\circ}\,.
\]
Then for every~$g\in\Gamma\setminus\Gamma_\infty$ we have
\[
\underiso{g}\cap\Lambda(\Gamma)\ne\varnothing\,.
\]
\end{lemma}

\begin{proof}
First note that, since the transformations in~$\PSLR$ extend smoothly to~$\overline{\H}^\geo$, it follows from~\eqref{ITEM:iso:eqn} that
\begin{equation}\label{EQ:mapunderiso}
g\act\underiso{g}=\wh\R\setminus\overline{\underiso{g^{-1}}}\qquad\text{and}\qquad g\act\bigl(\wh\R\setminus\underiso{g}\bigr)=\overline{\underiso{g^{-1}}}\,.
\end{equation}

Since~$\Lambda(\Gamma)$ is invariant under transformations in~$\Gamma$, it suffices to find, for every~$g\in\Gamma$, a pair~$(x,q)\in\Lambda(\Gamma)\times\Gamma$ such that~$q\act x\in\underiso{g}$.
For~$g\in\Gamma$ hyperbolic~$\underiso{g}$ contains the repelling fixed point of~$g$, which already implies the assertion in this case.

Let~$g\in\Gamma$ be parabolic and let~$x\in\Lambda(\Gamma)$ be not the fixed point of~$g$.
Note that for parabolic transformations~$p\in\Gamma$ and~$n\in\N$ we always have~$\iso{p^{n+1}}\subseteq\intiso{p^n}$ as well as the convergence of~$\iso{p^n}$ to the fixed point of~$p$ (in the sense that both elements of~$\geo\iso{p^n}$ converge to~$p$ as~$n\to\infty$).
Hence, if~$x\in\wh\R\setminus\underiso{g^{-1}}$, then, because of~\eqref{ITEM:iso:eqn}, we obtain~$g^{-2}\act x\in\underiso{g}$.
If~$x\in\underiso{g^{-1}}$ we find~$n\in\N$ such that~$x\in\underiso{g^{-n+1}}\setminus\underiso{g^{-n}}$.
Again,~\eqref{ITEM:iso:eqn} yields~$g^{-n-1}\act x\in\underiso{g}$.

Finally, assume that~$g\in\Gamma$ is elliptic of some order~$\sigma\in\N\setminus\{1\}$.
For~$\sigma=2$, i.e.,~$g$ an involution, we have~$\iso{g}=\iso{g^{-1}}$ and the assertion is obvious from~\eqref{ITEM:iso:eqn}.
Thus, assume~$\sigma>2$.
Then the isometric spheres of~$g$ and~$g^{-1}$ intersect without coinciding and for~$q\in\Gamma$ elliptic we define~$\xi(q)\in\R$ via the rule
\[
\{\xi(q)\}=\geo\iso{q}\setminus\underiso{q^{-1}}\,,
\]
which is obviously well-defined.
We further define~$\xi'(q)\in\R$ to be the unique point such that
\[
\{\xi(q),\xi'(q)\}=\geo\iso{q}\,.
\]
Then
\begin{equation}\label{EQN:elltranspoint}
q\act\xi(q)=\xi(q^{-1})\qquad\text{and}\qquad q\act\xi'(q)=\xi'(q^{-1})\,.
\end{equation}
By renaming~$g$ to~$g^{-1}$ if necessary we may assume the ordering
\begin{equation}\label{EQ:xiordering}
\xi(g)<\xi'(g^{-1})<\xi'(g)<\xi(g^{-1})\,.
\end{equation}
By~\cite[Theorem~2.4.3]{Katok_fuchsian} the only elementary Fuchsian groups are either cyclic or conjugate to a group generated by~$z\mapsto-\tfrac{1}{z}$ and~$z\mapsto\lambda z$ for~$\lambda>1$.
Hence, by assumption, there is no constellation in which~$\xi(g),\xi(g^{-1})$ are the only limit points of~$\Gamma$, and~$\Gamma$ is non-elementary. Therefore, we may choose
\begin{equation}\label{EQN:choosex}
x\in\Lambda(\Gamma)\setminus\bigcup_{i=1}^{\sigma-1}g^i\act\{\xi(g)\}\,.
\end{equation}
Define
\begin{align*}
\mathscr{V}\coloneqq&\left(\geo(\extiso{g}\cap\extiso{g^{-1}})\right)^{\circ}\\
=&\,(\max\{\xi(g),\xi(g^{-1})\},+\infty)_\R\cup\{\infty\}\cup(-\infty,\min\{\xi(g),\xi(g^{-1})\})_\R\,.
\end{align*}
Then~\eqref{ITEM:iso:eqn} implies
\begin{equation}\label{EQN:elltransset}
g^{-1}\act\mathscr{V}\subseteq\underiso{g}\,,
\end{equation}
and thus,~$g^{-1}\act\mathscr{V}\cap\mathscr{V}=\varnothing$.
But an application of~\eqref{EQN:elltranspoint} with~$q=g^{-1}$ implies that the set~$g^{-1}\act\left(\mathscr{V}\cup\{\xi(g^{-1})\}\right)\cup\mathscr{V}$ is an interval in~$\wh\R$.
Iterating this argument with~$q=g^{-2},g^{-3},\dots$ yields the (not necessarily disjoint) decomposition
\[
\wh\R=\bigcup_{i=1}^{\sigma-1}g^i\act\bigl(\mathscr{V}\cup\{\xi(g)\}\bigr)\,.
\]
Hence, for every choice of~$x$ obeying~\eqref{EQN:choosex} we find~$i\in\{1,\dots,\sigma-1\}$ such that~$g^i\act x\in\mathscr{V}$. Combining this with~\eqref{EQN:elltransset} yields the assertion for~$g$ being elliptic.
\end{proof}

With the notion of isometric spheres at hand we define the \emph{common exterior}
\[
\kund\coloneqq\bigcap_{\I\in\Iso{\Gamma}}\mathrm{ext}\,\I=\bigcap_{g\in\Gamma\setminus\Gamma_{\infty}}\extiso{g}\,.
\]
This is a convex set with its boundary in~$\H$ given by non-vertical geodesic segments contained in isometric spheres.
Under certain assumptions, it naturally contains fundamental domains for~$\Gamma$ as subsets:
For instance, if~$\Orbi$ has cusps and~$\infty$ represents a cusp of~$\Orbi$, then there exist real numbers~$r\in\R$ for which the set~$\fund(r)\coloneqq\fund_{\infty}(r)\cap\kund$ is a geometrically finite, convex, exact fundamental polyhedron for~$\Gamma$ in~$\H$ (see \cite{Ford_book, Ford} as well as \cite[Theorem~6.1.38]{Pohl_diss}), where
\[
\fund_{\infty}(r)\coloneqq\pr^{-1}((r,r+\lambda)_\R)
\]
is a fundamental domain for~$\Gamma_\infty$.
If~$\Orbi$ does not have cusps, then~$\kund$ itself is already a fundamental polyhedron for~$\Gamma$ (see also Proposition~\ref{PROP:kundisfund} below).
In both cases we call fundamental domains arising in this way \emph{Ford fundamental domain} or \emph{fundamental domain of the Ford type}.
The boundary structure of Ford fundamental domains enjoys various special properties.
One of particular importance for our constructions is stated in the following result.

\begin{lemma}\label{LEM:vertexcycles}
Let~$\Gamma$ be a geometrically finite Fuchsian group and let~$\fund$ be a Ford fundamental domain for~$Gamma$.
Let~$z\in\partial_\geo\fund$.
Then for all~$w\in\Gamma\act z\cap\partial_\geo\fund$ we have
\[
\Ima{w}=\Ima{z}\,.
\]
\end{lemma}

\begin{proof}
Since~$\Gamma$ fixes the upper half-plane as well as~$\wh\R$ respectively, it suffices to consider the boundary of~$\fund$ in~$\H$.
It further suffices to consider the case that~$\Orbi$ has cusps, because the proof in that case already includes all the arguments necessary for the case without cusps.
Without loss of generality we may assume that~$\infty$ represents a cusp of~$\Orbi$.
(If~$\Orbi$ is void of cusps we assume instead that a neighborhood of~$\infty$ is contained in the ordinary set~$\Omega(\Gamma)$.)
Since~$\fund$ is a Ford fundamental domain we can write
\begin{equation}\label{EQN:representwund}
\fund=\fund_\infty(r)\cap\bigcap_{g\in\Gamma\setminus\Gamma_\infty}\extiso{g}\,,
\end{equation}
where~$r\in\R$ and~$\fund_\infty(r)$ is as before.

Let~$z=x+\i y\in\partial\fund$ and let~$h\in\Gamma$.
There are three possibilities for the interrelation of~$z$ and~$h$: either~$h\in\Gamma_\infty$,~$h\in\Gamma\setminus\Gamma_\infty$ and~$z\in\iso{h}$, or~$h\in\Gamma\setminus\Gamma_\infty$ and~$z\notin\iso{h}$.
In the first case we have~$h=t_\lambda^m$, with~$m\in\Z$ and~$t_\lambda$ as before, with~$\lambda>0$ the cusp width of~$\pi(\infty)$.
Then
\[
\Ima(h\act z)=\Ima(x+\i y + m\lambda)=y=\Ima{z}\,.
\]
In the second case we obtain~$\Ima(h\act z)=\Ima{z}$ immediately from~\eqref{ITEM:iso:height}.
This leaves the third case.
From~\eqref{EQN:representwund} we obtain
\begin{equation}\label{EQN:partialwundin}
\begin{aligned}
\partial\fund&\subseteq\biggl(\partial\fund_\infty(r)\cup\partial\Bigl(\bigcap_{g\in\Gamma\setminus\Gamma_\infty}\extiso{g}\Bigr)\biggr)\setminus\complement\Bigl(\overline{\bigcap_{g\in\Gamma\setminus\Gamma_\infty}\extiso{g}}\Bigr)\\
&\subseteq\overline{\bigcap_{g\in\Gamma\setminus\Gamma_\infty}\extiso{g}}\,.
\end{aligned}
\end{equation}
Since~$z\in\partial\fund\setminus\iso{h}$ and~$\partial(\extiso{h})=\iso{h}$, relation~\eqref{EQN:partialwundin} implies that~$z\in\extiso{h}$.
From~\eqref{ITEM:iso:eqn} we obtain
\[
h\act z\in\intiso{h^{-1}}\subseteq\bigcup_{g\in\Gamma\setminus\Gamma_\infty}\intiso{g}=\complement\biggl(\overline{\bigcap_{g\in\Gamma\setminus\Gamma_\infty}\extiso{g}}\biggr)\,,
\]
where the identity is due to~\cite[Proposition~3.12]{Pohl_isofunddom}.
Comparing this to the relation in~\eqref{EQN:partialwundin} yields~$h\act z\notin\partial\fund$ and thereby finishes the proof.
\end{proof}

An isometric sphere~$\I\in\Iso{\Gamma}$ is called \emph{relevant}, if it contributes non-trivially to the boundary of~$\kund$ in~$\H$, that is, if~$\I\cap\partial\kund$ contains more than one point.
We denote by~$\REL{\Gamma}$ the subset of~$\Iso{\Gamma}$ of all relevant isometric spheres.
We further denote by~$\Gamma_{\RELL}$ the subset of~$\Gamma$ of generators of isometric spheres in~$\REL{\Gamma}$.

\begin{enumerate}[label=$(\mathrm{ISO}\arabic*)$,ref=$\mathrm{ISO}\arabic*$]
\setcounter{enumi}{7}
\item\label{ITEM:iso:relpart}
Let~$g\in\Gamma_{\RELL}$.
Then there exist~$\xi_1,\xi_2\in\overline{\H}^{\geo}$ such that
\[
\overline{\iso{g}}^{\geo}\cap\partial_{\geo}\kund=[\xi_1,\xi_2]_{\H}\,.
\]
Furthermore, also~$g^{-1}\in\Gamma_{\RELL}$ and we have
\[
\overline{\iso{g^{-1}}}^{\geo}\cap\partial_{\geo}\kund=[g\act\xi_1,g\act\xi_2]_{\H}\,.
\]
(This is proven analogously to \cite[Proposition~6.1.29]{Pohl_diss}.)
\end{enumerate}

\section{Cusp expansion}\label{SEC:cuspexp}

In this section we review the cusp expansion algorithm by Pohl~\cite{Pohl_diss}.
The cross sections it produces are seen to come from sets of branches.
This comes as no surprise, since Pohl's algorithm has been the starting point of our studies.
The notion of sets of branches has been introduced in order to identify the key aspects of her approach and subsequently generalize her results to a broader class of Fuchsian groups as well as a wider variety of suitable cross sections.
Besides stating the main result about the attainment of a set of branches, in this section we further collect several properties specific to the cusp expansion method.
We refer the reader to~\cite{Pohl_diss} for the explicit constructions and results regarding these properties (exact references are added to each of them).

As the name suggests, the essential restriction the cusp expansion algorithm faces is that the orbisurface in question is required to have cusps, or, equivalently, that its fundamental group~$\Gamma$ contains parabolic elements.
We will demonstrate how to overcome this restriction in Section~\ref{SEC:cuspless}, but for the remainder of this section we assume~$\Gamma$ to obey it.
Since all of our considerations are invariant under conjugation in~$\PSLR$, we may assume without loss of generality that~$\infty$ represents a cusp of~$\Orbi$.
As mentioned in Section~\ref{SEC:isoms}, the stabilizer subgroup~$\Gamma_\infty$
of~$\infty$ in~$\Gamma$ is then cyclic and generated by~$t_{\lambda}=\begin{bsmallmatrix}1&\lambda\\0&1\end{bsmallmatrix}$, for~$\lambda>0$ the unique cusp width of~$\pi(\infty)$.

For every~$\I\in\Iso{\Gamma}$ there exists a unique point~$\summit{\I}\in\H$ of maximal height, called the \emph{summit} of~$\I$.
For~$g=\begin{bsmallmatrix}1&\lambda\\0&1\end{bsmallmatrix}\in\Gamma\setminus\Gamma_\infty$ we have~$c\ne0$ and
\begin{equation}\label{EQDEF:summit}
\summit{\iso{g}}=-\frac{d}{c}+\frac{1}{\abs{c}}\,.
\end{equation}

The cusp expansion algorithm presumes~$\Gamma$ to fulfill the following condition:
\begin{enumerate}[label=$(\mathrm{A})$,ref=$\mathrm{A}$]
\item\label{condA}
For every~$\I\in\REL{\Gamma}$ its summit $\summit{\I}$ is contained in the geodesic segment~$\I\cap\partial\kund$ but not an endpoint of it.
\end{enumerate}
In \cite[Section 6.3]{Pohl_diss} it is shown that this is a proper restriction, thus we have to assume it here as well.
However, we make no additional use of it here.\footnote{In fact, we go to some lengths in order to avoid use of~\eqref{condA}.
Lemma~\ref{LEM:thirdsphere} as well as large parts of the proof of Proposition~\ref{LEM:Anotempty} below are necessary solely to that end.}
That means that, once (a modification of) the cusp expansion algorithm has been shown to work regardless of this restriction (which is conjectured to be feasible), it can safely be removed from here as well.
As long as all statements of this section remain valid for the modification of the algorithm, no further adjustments are required.

Given~\eqref{condA}, the properties of~$\kund$ allow for the construction of a set
\[
\BrS_{\P}\coloneqq\{\Cs{\P,1},\dots,\Cs{\P,N}\}\,,
\]
with a certain~$N\in\N$ and subsets~$\Cs{\P,j}$ of the unit tangent bundle~$\UTB\H$, such that~$\BrU_{\P}\coloneqq\bigcup\BrS_{\P}$ is a representative for a cross section for the geodesic flow on~$\Orbi$.
In the notation of \cite{Pohl_diss} we have
\begin{align}\label{EQN:inPnota}
\BrS_{\P}=\defset{\mathrm{CS}'(\widetilde{\mathcal{B}})}{\widetilde{\mathcal{B}}\in\widetilde{\mathbb{B}}_{\mathbb{S},\mathbb{T}}}\,,
\end{align}
where the subscript~$\mathbb{S}$ fixes a sequence of choices to be made during the construction of these sets, and the subscript~$\mathbb{T}$ indicates that arbitrary translations of the sets~$\Cs{\P,j}$ by elements of~$\Gamma_{\infty}$ are permitted and a collection of such translations is chosen and applied.
The following statements of this section are meant to be understood \emph{``for all possible choices of~$\mathbb{S}$ and~$\mathbb{T}$''}.
We refer to \cite[Sections 6.2--6.7.1]{Pohl_diss} for the exact construction of these objects.

The proof of the following result is straightforward.
For the convenience of the reader we included it in the Appendix, as well as the proof of Proposition~\ref{PROP:finram} below.

\begin{thm}\label{THM:cuspexpSoB}
$\BrS_{\P}$ is a set of branches for~$\widehat{\Phi}$.
\end{thm}

In \cite[Section~3.9]{Pohl_Wab} we introduced and studied the notion of \emph{branch ramification}.
A set of branches~$\BrS$ is called \emph{infinitely ramified} if there exists a pair~$(j,k)\in A\times A$ such that~$\#\Trans{}{j}{k}=+\infty$.
If such a pair does not exist, then~$\BrS$ is called ~\emph{finitely ramified}.
This turned out to be an important distinction, since the strict transfer operator approach presumes certain sets of transformations of finite cardinality and these sets are directly informed by the transition sets~$\Trans{}{.}{.}$.

\begin{prop}\label{PROP:finram}
$\BrS_{\P}$ is finitely ramified.
\end{prop}

In the remainder of this section we collate some properties which we will require in the upcoming proofs.
Denote by~$\mathrm{c}(\I)$ the center of an isometric sphere~$\I$ and define
\[
\mathscr{Q}\coloneqq\geo\kund\cup\defset{\mathrm{c}(\I)}{\I\in\REL{\Gamma}}\subseteq\R\setminus\R_{\st}\,.
\]
Then the following statements are easily obtained from~\cite{Pohl_diss}.
\begin{enumerate}[label=$(\mathrm{P}\arabic*)$,ref=$\mathrm{P}\arabic*$]
\item\label{ITEM:cex:finiteset}
The set~$\widetilde{\mathbb{B}}_{\mathbb{S},\mathbb{T}}$ is finite. Hence, the integer~$N$ does indeed exist and we may write~$A=\{1,\dots,N\}$ as before as well as~$\widetilde{\mathbb{B}}_{\mathbb{S},\mathbb{T}}=\{\widetilde{\mathcal{B}}_1,\dots,\widetilde{\mathcal{B}}_N\}$.
(\cite[Theorem~6.2.20]{Pohl_diss} and subsequent constructions)

\item\label{ITEM:cex:setstructure}
For each~$j\in A$ we have~$\Cs{\P,j}=\mathrm{CS}'(\widetilde{\mathcal{B}}_j)\subseteq\UTB\H$ and there exists a unique point~$x_j\in\mathscr{Q}$ such that~$\base{\Cs{\P,j}}=\pr^{-1}(x_j)=(x_j,\infty)_{\H}$.
(\cite[Propositions~6.4.11--6.4.13 and~6.6.21--6.6.23]{Pohl_diss})

\item\label{ITEM:cex:vectors}
All vectors of~$\Cs{\P,j}$ point into the same open half-space relative to~$(x_j,\infty)_{\H}$, which we may denote by~$\Plusspp{j}$.
We further set~$\Minusspp{j}\coloneqq\H\setminus\overline{\Plusspp{j}}$ and denote by~$I_{\P,j}$ and~$J_{\P,j}$ the largest open set contained in~$\geo\Plusspp{j}$ and~$\geo\Minusspp{j}$ respectively.
Then for every pair~$(x,y)\in I_{\P,j}\times J_{\P,j}$ there exists exactly one~$\nu\in\Cs{\P,j}$ such that~$(\gamma_\nu(+\infty),\gamma_\nu(-\infty))=(x,y)$.
(\cite[Lemma~6.7.9]{Pohl_diss})

\item\label{ITEM:cex:tessellation}
There exists an exact tessellation (in the sense of an essentially disjoint union) of~$\H$ by a set~$\mathscr{B}$ of convex closed hyperbolic polyhedra and (vertical) strips, i.e., sets of the form~$\pr^{-1}([a,b]_\R)$ for~$a,b\in\R$, such that
\begin{enumerate}[label=$\mathrm{(\alph*)}$,ref=$\mathrm{\alph*}$]
\item\label{ITEM:tessellation:structure}
for every~$B\in\mathscr{B}$ the boundary of~$B$ in~$\H$ decomposes disjointly into finitely many complete geodesic segments, called~\emph{sides} of~$B$, of which at most two are vertical, and for every side~$b$ of~$B$ there exist exactly two distinct pairs~$(j_1,g_1),(j_2,g_2)\in A\times\Gamma$ such that
\[
b=g_i\act\base{\Cs{\P,j_i}}\,,
\]
\item\label{ITEM:tessellation:strip}
each~$B\in\mathscr{B}$ is either a strip, i.e., of the form~$B=g\act\pr^{-1}(I)$ for some finite closed interval~$I\subseteq\R$ and~$g\in\Gamma$, or its boundary~$\geo B$ in~$\partial_{\geo}\H$ consists of finitely many points all of which are contained in~$\Gamma\act\mathscr{Q}$,
\item\label{ITEM:tessellation:triangle}
if~$B\in\mathscr{B}$ is neither a strip nor a hyperbolic triangle, then every side of~$B$ is of the form~$(g^{k}\act\infty,g^{k+1}\act\infty)_{\H}$, for some elliptic transformation~$g\in\Gamma$ and~$k\in\{0,\dots,\ord(g)-1\}$,
\item\label{ITEM:tessellation:include}
for each~$j\in A$ and every~$g\in\Gamma$ there exits~$B\in\mathscr{B}$ such that~$g\act\base{\Cs{\P,j}}$ is a side of~$B$.
\end{enumerate}
(\cite[Propositions~6.6.21--6.6.23 and 6.7.10, Corollaries~6.4.18~and~6.6.24, Construction~6.4.2, and Proposition~6.4.4]{Pohl_diss})

\item\label{ITEM:cex:uniqueness}
For all~$j\in A$ there exists a pair~$(k,g)\in A\times\Gamma$ such that~$\mathrm{H}_{\pm}^{\P}(j)=g\act\mathrm{H}_{\mp}^{\P}(k)$.
If~$\Plusspp{j}=g\act\Plusspp{k}$ for~$j,k\in A$ and~$g\in\Gamma$, then~$j=k$ and~$g=\id$.
(\cite[Proposition~6.7.10]{Pohl_diss})

\item\label{ITEM:cex:nextinter}
Let~$\nu\in\BrU_{\P}$ be such that~$(\gamma_\nu(+\infty),\gamma_\nu(-\infty))\subseteq\R_{\st}\times\R_{\st}$.
Then both the values
\begin{align*}
t_+(\nu)&\coloneqq\min\defset{t>0}{\gamma_\nu'(t)\in\Gamma\act\BrU_{\P}}
\intertext{and}
t_-(\nu)&\coloneqq\max\defset{t<0}{\gamma_\nu'(t)\in\Gamma\act\BrU_{\P}}
\end{align*}
are finite.
(\cite[Proposition~6.7.12]{Pohl_diss})
\item\label{ITEM:cex:summits}
For every~$\I\in\REL{\Gamma}$ for which its summit~$\summit{\I}$ is contained in~$\I\cap\partial\kund$, there exists~$(j,g)\in A\times\Gamma$ such that~$g\act\base{\Cs{\P,j}}$ is vertical and~$\summit{\I}\in g\act\base{\Cs{\P,j}}$.
Then either~$g\in\Gamma_\infty$ or~$\iso{g}\in\REL{\Gamma}$. (The first statement follows from the construction of the sets~$B\in\mathscr{B}$, see~\cite[Section 6.4]{Pohl_diss}. The second statement is easily derived from Lemma~\ref{LEM:centerofREL} below.)
\end{enumerate}

We further require the following observation.

\begin{lemma}\label{LEM:centerofREL}
Let~$j\in A$. For~$g\in\Gamma$ with~$g\act x_j=\infty$ we have~$\iso{g}\in\REL{\Gamma}$.
\end{lemma}

\begin{proof}
Since~$x_j\in\R$ for every~$j\in A$, we have~$g\in\Gamma\setminus\Gamma_\infty$.
Hence,~$\iso{g}$ is well-defined with center~$x_j=g^{-1}\act\infty$.
Because of~$x_j\in\mathscr{Q}$ (see~\eqref{ITEM:cex:setstructure}) the point~$x_j$ is also the center of some relevant isometric sphere.
Lemma~\ref{LEM:isoconcentric} now yields the assertion.
\end{proof}

We close this section with a technical observation that will come in handy in the upcoming constructions.
The cusp represented by the point at infinity imposes that the geometric structures induced by~$\Gamma$ are periodic in directions parallel (in the Euclidean sense) to~$\partial_{\geo}\H$ with period length equal to the cusp width of~$\pi(\infty)$.
This means that, for instance, we have
\begin{equation}\label{EQ:gammainftykund}
t_{\lambda}^q\act\kund = \kund\qquad\text{and}\qquad t_{\lambda}^q\act\bigcup\Iso{\Gamma}=\bigcup\Iso{\Gamma}
\end{equation}
for all~$q\in\Z$.
This periodicity is also inherited by~$\base{\BrU_{\P}}$, which is implied by the following result, which, in turn, is an immediate consequence of~\eqref{ITEM:cex:setstructure} and the tessellation property of~$\fund_{\infty}(.)$.

\begin{lemma}\label{LEM:CPperiod}
For every~$r\in\R$ there exist~$i_1,\dots,i_N\in\Z$ such that
\[
\bigcup_{j=1}^nt_{\lambda}^{i_j}\act\base{\Cs{j}}\subseteq\pr^{-1}([r,r+\lambda]_\R)\,.
\]
\end{lemma}

\section{Sets of branches for orbisurfaces without cusps}\label{SEC:cuspless}

We now, for the moment, abandon the orbisurfaces with cusps to study the situation when only funnels are present.
Thus, we assume that~$\Gamma$ is a geometrically finite, non-cocompact Fuchsian group that contains hyperbolic but no parabolic elements and for which the orbit space~$\Orbi=\quod{\Gamma}{\H}$ is of infinite volume.
By conjugation in~$\PSLR$ we can always achieve that
\begin{enumerate}[label=$(\mathrm{\star})$,ref=$\mathrm{\star}$]
\item\label{inftynei}
the ordinary set~$\Omega(\Gamma)=\partial_{\geo}\H\setminus\Lambda(\Gamma)$ contains a neighborhood of~$\infty$.
\end{enumerate}
Since the statements that follow are invariant under conjugation, we may assume that this is the case.
Then the stabilizer subgroup~$\Gamma_{\infty}=\Stab{\Gamma}{\infty}$ is trivial.
We define the set of isometric spheres~$\Iso{\Gamma}$, the set
\[
\kund=\bigcap_{\I\in\Iso{\Gamma}}\mathrm{ext\,}\I\,,
\]
as well as the sets ~$\REL{\Gamma}$ and~$\Gamma_{\RELL}$ of relevant isometric spheres and their generators as in Section~\ref{SEC:isoms}.
Then, as mentioned there,~$\kund$ is already a fundamental domain for~$\Gamma$ (see also Proposition~\ref{PROP:kundisfund} below).
We further remark that the assumptions imposed on~$\Gamma$ assure that~$\Iso{\Gamma}\ne\varnothing$.

As already described in the introduction, our strategy is as follows: We construct a new Fuchsian group~$\Gamma_{\wund}$, called the auxiliary group, from~$\Gamma$ via a cut-off procedure on the fundamental domain~$\kund$.
The group~$\Gamma_{\wund}$ then has a cusp represented by~$\infty$ and thus, under the additional assumption that~$\Gamma$ fulfills~\eqref{condA}, will allow for an application of the cusp expansion algorithm.
This, as we have seen in Theorem~\ref{THM:cuspexpSoB}, yields a cross section that emerges from a set of branches~$\BrS_{\wund}$.
We then return to~$\Gamma$ and see that~$\BrS_{\wund}$ induces a set of branches for the geodesic flow on~$\Orbi$ as well.

The following two statements regarding isometric spheres are specific for the situation without cusps.
\begin{enumerate}[label=$(\mathrm{ISO}\arabic*)$,ref=$\mathrm{ISO}\arabic*$]
\setcounter{enumi}{8}
\item\label{ITEM:iso:biject}
The map
\[
\Gamma\setminus\{\id\}\ni g\mapstoo\iso{g}\in\Iso{\Gamma}
\]
is a bijection. (Combine the first part of~\eqref{ITEM:iso:gener} with~$\Gamma_\infty=\{\id\}$.)

\item\label{ITEM:iso:slice}
There exist~$\alpha,\beta\in\R$,~$\alpha<\beta$, such that
\[
\bigcup\Iso{\Gamma}\subseteq\pr^{-1}([\alpha,\beta]_\R)\,.
\]
(See \cite[pp. 114--115]{Lehner}, or, equivalently, Remark~\ref{REM:finsides} below.)
\end{enumerate}

From these properties one easily infers that~$\kund$ is a fundamental domain for~$\Gamma$.
This is already known by virtue of the work of Ford~\cite{Ford_book,Ford} and Pohl~\cite{Pohl_isofunddom}, but we included a proof in the Appendix for the convenience of the reader.

\begin{prop}\label{PROP:kundisfund}
The set~$\kund$ is a convex, geometrically finite, exact fundamental polyhedron for~$\Gamma$ in~$\H$.
\end{prop}

\begin{remark}\label{REM:finsides}
In~\cite[Section~12.4]{Ratcliffe} it is shown that, in two dimensions, a convex polyhedron is geometrically finite if and only if it has a finite number of sides.
For~$g\in\Gamma_{\RELL}$ denote by~$b_g$ the geodesic segment from~\eqref{ITEM:iso:relpart}.
Since~$g\mapsto b_g$ is a bijection between~$\Gamma_{\RELL}$ and the set of sides of~$\kund$ by~\eqref{ITEM:iso:biject}, Proposition~\ref{PROP:kundisfund} together with~\eqref{ITEM:iso:biject} implies that the sets~$\Gamma_{\RELL}$ and~$\REL{\Gamma}$ are both finite.
\end{remark}

We now briefly recall the concept of (finite and infinite) vertex cycles and angle sums.
For a comprehensive analysis of this subject we refer the reader to~\cite{Lehner}.

Let~$\fund$ be an exact, convex fundamental polyhedron for some geometrically finite Fuchsian group~$\Lambda$.
A \emph{finite vertex} of $\fund$ is a point $v\in\H$ that is the common endpoint of two distinct sides of $\fund$.
(A \emph{side} of~$\fund$ is a geodesic segment of positive length of the form~$\overline{\fund}\cap g\act\overline{\fund}$, for~$g\in\Gamma$.)
Denote by $V_{\fund}$ the set of finite vertices of $\fund$.
For $v\in V_{\fund}$ its \emph{cycle} $C(v)$ is defined as the intersection of $\overline{\fund}^{\geo}$ with $\Gamma\act v$.
The cycle $C(v)$ is a finite set and consists solely of vertices of $\fund$, say $C(v)=\{v_1,\dots,v_k\}$.
For each $j$ denote by $\theta_j$ the angle which $\fund$ subtends at $v_j$.
We define the \emph{angle sum} of $C(v)$ to be
\[
\theta(C(v))\coloneqq\sum_{j=1}^k\theta_j.
\]
Combination of Proposition~\ref{PROP:kundisfund} with~\cite[Theorem~IV.4C]{Lehner} yields the following result.

\begin{cor}\label{COR:anglesum}
For all $v\in V_{\kund}$ there exists~$\omega=\omega(v)\in\N$ such that
\[
\theta(C(v))=\frac{2\pi}{\omega}.
\]
\end{cor}

An~\emph{infinite vertex} of~$\fund$ is the common endpoint of two distinct sides of~$\fund$ contained in~$\partial_{\geo}\H$.
We denote the set of infinite vertices of~$\fund$ by~$V_{\fund}^{\geo}$.
The cycles~$C^{\geo}(.)$ in~$V_{\fund}^{\geo}$ are defined analogously as in the finite case and consist of finitely many infinite vertices of~$\fund$.
Concatenating the side-pairing transformations associated to the vertices in the vertex cycle~$C^{\geo}(v)$ in the order imposed by the cycle yields one of the two~\emph{cycle transformations}~$c_v,c'_v$, where we have~$c'_v=c_v^{-1}$.
This construction is also possible for finite vertices but we only require it here in the infinite case.
Note that~$V_{\kund}^{\geo}=\varnothing.$

\subsection*{An auxiliary group.}

We now construct the Fuchsian group to which we can apply the cusp expansion algorithm.
Let~$\alpha\in\R$ be the largest and~$\beta\in\R$ be the smallest number for which
\[
\partial\kund\subseteq\pr^{-1}([\alpha,\beta]_\R)
\]
(see~\eqref{ITEM:iso:slice}).
We fix~$\alpha'\in(-\infty,\alpha)_\R$ and~$\beta'\in(\beta,+\infty)_\R$, and set
\begin{align}\label{EQDEF:lambdachoice}
\lambda\coloneqq\beta'-\alpha'\,.
\end{align}
As before we set~$t_{\lambda}\coloneqq\begin{bsmallmatrix}1&\lambda\\0&1\end{bsmallmatrix}\in\PSLR$.
Now we define
\begin{align}\label{EQDEF:wund}
\wund\coloneqq\kund\cap\pr^{-1}((\alpha',\beta')_\R)\subseteq\H
\end{align}
The set~$\wund$ will play the role of the fundamental domain for our new group of isometries.
We can read off several properties of~$\wund$ from Proposition~\ref{PROP:kundisfund} already:
It is again a convex polyhedron in~$\H$ whose set of sides is given by
\[
\defset{b_g}{g\in\Gamma_{\RELL}}\cup\{[\alpha',\infty]_{\H},[\beta',\infty]_{\H}\}\,,
\]
where~$b_g$ denotes the geodesic segment from~\eqref{ITEM:iso:relpart}.
From Remark~\ref{REM:finsides} we infer that this is a finite set and hence,~$\wund$ is geometrically finite.
It further inherits the side-pairing from~$\kund$, extended by the transformations~$t_{\lambda}^{\pm1}$ on its vertical sides.
Again, the side-pairing is exact.

Recall the notion of Poincar\'e polygons in the form of~\cite{Maskit}.

\begin{prop}\label{PROP:wundispopoly}
The domain~$\wund$ is a Poincar\'e polygon.
\end{prop}

\begin{proof}
Following~\cite{Maskit} it remains to show that~$\wund$ has the following two properties:
\begin{enumerate}[label=$\mathrm{(\roman*)}$, ref=$\mathrm{\roman*}$]
\item\label{vertexcondition} For each~$v\in V_{\wund}$ there exists an integer~$\omega\ne0$ so that~$\theta(C(v))=\frac{2\pi}{\omega}$.

\item\label{paraboliccuspcondition} For each~$v\in V_{\wund}^{\geo}$ the cycle transformations~$c_v,c'_v$ are parabolic.
\end{enumerate}
We have~$V_{\wund}= V_{\kund}$.
Thus,~\eqref{vertexcondition} follows from Corollary~\ref{COR:anglesum}.
The only infinite vertex of~$\wund$ is~$v=\infty$, for which~$\{c_v,c'_v\}=\{t_{\lambda},t_{\lambda}^{-1}\}$.
This shows~\eqref{paraboliccuspcondition}.
\end{proof}

Proposition~\ref{PROP:wundispopoly} combined with the Poincar\'e theorem for fundamental polygons (in the form of \cite{Maskit}) implies that $\wund$ is a fundamental domain for the group generated by its side-pairing transformations, which we denote by $\Gamma_{\wund}$, and that $\Gamma_{\wund}$ is Fuchsian.
Since $\wund$ is geometrically finite, convex, and exact, the group $\Gamma_{\wund}$ is geometrically finite.
Denote by~$\Orbi_{\wund}$ the orbisurface arising as the orbit space of~$\Gamma_{\wund}$.
We see that~$\Orbi_{\wund}$ has a single cusp represented by~$\infty$.
All other hyperbolic ends of~$\Orbi_{\wund}$ match those of~$\Orbi$ in number and relative location.

Denote by~$\wh\R_{\st,\wund}$ the set~$\wh\R_{\st}$ with respect to~$\Gamma_\wund$, that is
\[
\wh\R_{\st,\wund}\coloneqq\Lambda(\Gamma_\wund)\setminus\Gamma_\wund\act\infty\,.
\]
The group~$\Gamma_\wund$ can be viewed as arising from~$\Gamma$ via the addition of the parabolic transformation~$t_\lambda$ to the set of generators.
Hence,~$\Gamma$ is a non-trivial subgroup of~$\Gamma_\wund$.
From this the following result is immediate.

\begin{cor}\label{COR:limitsetinclude}
We have~$\Lambda(\Gamma)\subseteq\Lambda(\Gamma_\wund)$,~$\wh\R_{\st}\subseteq\wh\R_{\st,\wund}$, and~$E(\Orbi)\subseteq E(\Orbi_\wund)$.
\end{cor}

In order to apply the cusp expansion algorithm to~$\Gamma_\wund$ we require the common exterior w.r.t.~$\Gamma_\wund$, i.e., the set
\begin{equation}\label{EQDEF:kundwund}
\kund_\wund\coloneqq\kund_{\Gamma_\wund}=\bigcap_{g\in\Gamma_\wund\setminus\Gamma_{\wund,\infty}}\extiso{g}\,,
\end{equation}
where~$\Gamma_{\wund,\infty}$ denotes the stabilizer of~$\infty$ in~$\Gamma_\wund$.
If~$\fund_\wund$ is a Ford fundamental domain for~$\Gamma_\wund$, then we can re-obtain~$\kund_\wund$ by means of the~$\Gamma_{\wund,\infty}$-invariance of~$\kund_\wund$ (see~\eqref{EQ:gammainftykund}).
To be more precise, from~$\fund_\wund=\kund_\wund\cap\fund_{\wund,\infty}$, for~$\fund_{\wund,\infty}$ a fundamental domain for~$\Gamma_{\wund,\infty}$, and the tessellation property we obtain
\[
\kund_\wund=\bigl(\,\Gamma_{\wund,\infty}\,\act\,\overline{\fund_\wund}\,\bigr)^\circ\,.
\]
We show that~$\wund$ is a Ford fundamental domain for~$\Gamma_\wund$, starting with the verification that~$\lambda$ from~\eqref{EQDEF:lambdachoice} is the cusp width of the one cusp of~$\Gamma_\wund$.

\begin{lemma}\label{LEM:Stabinwund}
The stabilizer~$\Gamma_{\wund,\infty}$ of~$\infty$ in~$\Gamma_{\wund}$ is generated by~$t_{\lambda}$.
\end{lemma}

\begin{proof}
Since~$t_{\lambda}$ is a side-pairing transformation for~$\wund$, we have~$t_{\lambda}\in\Gamma_{\wund}$.
Hence,~$\infty$ is a parabolic fixed point and every transformation in~$\Gamma_{\wund}$ fixing~$\infty$ must be parabolic.
Therefore, every non-identity element in~$\Gamma_{\wund,\infty}$ has the same fixed point set, which implies that~$\Gamma_{\wund,\infty}$ is cyclic.
Hence,~$\Gamma_{\wund,\infty}$ is generated by some element~$t_{\kappa}=\begin{bsmallmatrix}1&\kappa\\0&1\end{bsmallmatrix}$ with~$\abs{\kappa}\leq\lambda$.
Suppose~$\abs{\kappa}<\lambda$.
W.l.o.g. we may assume~$\kappa>0$.
Because of~\eqref{ITEM:iso:boundradii} and non-vertical sides of~$\wund$ coinciding with those of~$\kund$, there exists~$M>0$ such that
\[
\wund_{M}\coloneqq\defset{z\in\pr^{-1}((\alpha',\beta')_\R)}{\Ima{z}\geq M}\subseteq\wund\,.
\]
Let~$\eta\coloneqq\tfrac{\lambda-\kappa}{2}$ and~$y>M$.
Then~$z^*\coloneqq\alpha'+\eta+\i y\in\wund_{M}$.
But also
\begin{align*}
t_{\kappa}\act z^*&=\alpha'+\eta+\kappa+\i y=\alpha'+\lambda-\eta+\i y=\beta'-\eta+\i y\in\wund_{M}\,,
\end{align*}
which contradicts~$\wund$ being a fundamental domain for~$\Gamma_{\wund}$.
Thus,~$\abs{\kappa}=\lambda$, which completes the proof.
\end{proof}

Because of Lemma~\ref{LEM:Stabinwund}, the strip
\[
\fund_{\wund,\infty}\coloneqq\pr^{-1}((\alpha',\beta')_\R)
\]
is a fundamental domain for~$\Gamma_{\wund,\infty}$ in~$\H$.
As before we denote by~$\REL{\Gamma_{\wund}}$ the set of relevant isometric spheres of~$\Gamma_{\wund}$. Additionally, we denote by~$\RELL_{\wund}$ the subset of isometric spheres of~$\Gamma_{\wund}$ that contribute non-trivially to the boundary of~$\wund$.

\begin{lemma}\label{LEM:RelRell}
$\REL{\Gamma}=\RELL_{\wund}\,.$
\end{lemma}

\begin{proof}
By construction we have~$\Gamma_{\RELL}\subseteq\Gamma_{\wund}$.
Since the non-vertical sides of~$\wund$ coincide with those of~$\kund$, it follows that
\[
\REL{\Gamma}\subseteq\RELL_{\wund}\,.
\]
For~$\I\in\RELL_{\wund}$ also~$\I\cap\partial\kund$ contains more than one point, implying
\[
\REL{\Gamma}\supseteq\RELL_{\wund}\,.\qedhere
\]
\end{proof}

\begin{prop}\label{PROP:wundisFord}
The fundamental domain $\wund$ for $\Gamma_{\wund}$ is of the Ford type.
\end{prop}

\begin{proof}
Because of Lemma~\ref{LEM:Stabinwund} it remains to show that
\[
\wund=\fund_{\wund,\infty}\cap\bigcap_{g\in\Gamma_{\wund}\setminus\Gamma_{\wund,\infty}}\extiso{g}=\fund_{\wund,\infty}\cap\bigcap_{\I\in\REL{\Gamma_{\wund}}}\mathrm{ext\,}\I\,.
\]
From Lemma~\ref{LEM:RelRell} we obtain
\begin{align*}
\wund&=\pr^{-1}((\alpha',\beta')_\R)\cap\kund=\fund_{\wund,\infty}\cap\bigcap_{\I\in\REL{\Gamma}}\mathrm{ext\,}\I\\
&=\fund_{\wund,\infty}\cap\bigcap_{\I\in\RELL_{\wund}}\mathrm{ext\,}\I=\fund_{\wund,\infty}\cap\bigcap_{\I\in\REL{\Gamma_{\wund}}}\mathrm{ext\,}\I\,.\qedhere
\end{align*}
\end{proof}

\begin{cor}\label{COR:condA}
If~$\Gamma$ fulfills condition~\eqref{condA}, then so does~$\Gamma_{\wund}$.
\end{cor}

The results of this section show that, under the assumption that~$\Gamma$ fulfills~\eqref{condA}, the group~$\Gamma_{\wund}$ meets all the requirements for an application of the cusp expansion algorithm.
Hence, by virtue of Theorem~\ref{THM:cuspexpSoB} and Proposition~\ref{PROP:finram} we obtain a finitely ramified set of branches for the geodesic flow on~$\Orbi_{\wund}$, which here we denote as
\[
\BrS_{\wund}=\{\Cs{\wund,1},\dots,\Cs{\wund,N}\}\,.
\]
As before, we set~$A\coloneqq\{1,\dots,N\}$, define
\[
\Cs{\wund,j,\st}\coloneqq\defset{\nu\in\Cs{\wund,j}}{(\gamma_\nu(+\infty),\gamma_\nu(-\infty))\in\R_{\st}\times\R_{\st}}
\]
for every~$j\in A$, and set
\[
\BrU_{\P,\st}\coloneqq\bigcup_{j\in A}\Cs{\P,j,\st}\,.
\]
Because of Lemma~\ref{LEM:CPperiod} no loss of generality is entailed by assuming that
\begin{align}\label{EQN:SoBinfund}
\base{\Cs{\wund,j}}\subseteq\pr^{-1}([\alpha',\beta']_\R)
\end{align}
for all~$j\in A$.
So we assume that this is the case.
We further denote by~$\Trans{\wund}{j}{k}$ the transition set for $j,k\in A$ given by~\eqref{BP:intervaldecomp} as well as by~$I_{\wund,j}$ and~$J_{\wund,j}$ the intervals associated to~$\Cs{\wund,j}$ by~\eqref{BP:pointintohalfspaces}.
As before we denote by~$x_j$ the unique point associated to~$\Cs{\wund,j}$ by~\eqref{ITEM:cex:setstructure}.

\subsection*{A set of branches for~$\Gamma$}

We now transfer the set of branches~$\BrS_{\wund}$ back to the orbit space~$\Orbi$ of the initial group~$\Gamma$ whose hyperbolic ends are all funnels.
Not all of the branches~$\Cs{\wund,j}$ ``survive'' this transfer.
We clarify in the following what we mean by that.
Since~$\Gamma$ contains no parabolic elements and the ordinary set is assumed to contain a neighborhood of~$\infty$, we have
\[
\R_{\st}=\widehat\R_{\st}=\Lambda(\Gamma)\,.
\]
For~$j\in A$ define
\[
\Cs{\wund,j,\st}\coloneqq\defset{\nu\in\Cs{\wund,j}}{(\gamma_\nu(+\infty),\gamma_\nu(-\infty))\in\R_{\st}\times\R_{\st}}\,.
\]
Then~$\Cs{\wund,j,\st}=\varnothing$ whenever~$I_{\wund,j}\cap\R_{\st}=\varnothing$ or~$J_{\wund,j}\cap\R_{\st}=\varnothing$.
This is the case, for instance, if~$x_j\in\{\alpha',\beta'\}$.
Since the cusp expansion algorithm for~$\Gamma_{\wund}$ does indeed establish branches with that property and those do not intersect periodic geodesics of~$\Gamma$, it is necessary to exclude those from now on (see \eqref{BP:closedgeodesicsHtoX}).
Therefore, we define
\[
\neindex\coloneqq\defset{j\in A}{\Cs{\wund,j,\st}\ne\varnothing}\,.
\]
\begin{prop}\label{LEM:Anotempty}
$\neindex\ne\varnothing$.
\end{prop}

The proof of Proposition~\ref{LEM:Anotempty} makes use of Lemma~\ref{LEM:isolimit}.
There the case of groups conjugate to~$\left<z\mapsto-\tfrac{1}{z},z\mapsto\lambda z\right>$ had to be excluded, which we therefore have to treat separately here. This is done in the following example, which simultaneously illustrates the strategy of the ensuing proof: Utilizing Lemma~\ref{LEM:isolimit} and the density of the set~$E(\Orbi)$ from~\eqref{EQNDEF:EX} in~$\Lambda(\Gamma)\times\Lambda(\Gamma)$, we find interrelated hyperbolic fixed points underneath the outermost isometric spheres (i.e., in the sets~$\underiso{g_{1/2}}$ from Lemma~\ref{LEM:isolimit}, for generators~$g_{1/2}$ of said spheres).
We further identify a branch copy separating the hyperbolic fixed points, which is subsequently seen to be intersected by the associated hyperbolic axis.

\begin{example}\label{EX:excludedgroups}
We consider the conjugation of the aforementioned group by the transformation~$\tfrac{1}{\sqrt{2}}\begin{bsmallmatrix}1&-1\\1&1\end{bsmallmatrix}\in\PSLR$, which leads to the generators
\[
h_\lambda\coloneqq\frac{1}{2\sqrt{\lambda}}\begin{bmatrix}\lambda+1&\lambda-1\\\lambda-1&\lambda+1\end{bmatrix}\qquad\text{and}\qquad s\coloneqq\begin{bmatrix}0&-1\\1&0\end{bmatrix}\,,
\]
for~$\lambda>1$. A fundamental domain is indicated in Figure~\ref{FIG:excludedgroupsfund}.
\begin{figure}[h]
\begin{tikzpicture}[scale=2.3]
\tikzmath{\q=1/8;
				  \r=2*sqrt(\q)/abs(\q-1);
				  \c=(\q+1)/(1-\q);
				  \d=(\q+1)/(\q-1);
				  }
\fill[color=lightgray!50] (-2.7,0) -- (2.7,0) -- (2.7,2.3) -- (-2.7,2.3) -- cycle;
\fill[color=white] (-1,0) arc (180:0:1) -- cycle;
\draw (-1,0) arc (180:0:1);
\fill[color=white] (\d-\r,0) arc (180:0:\r) -- cycle;
\fill[color=white] (\c+\r,0) arc (0:180:\r) -- cycle;
\draw (\d-\r,0) arc (180:50:\r);
\draw (\c+\r,0) arc (0:130:\r);
\draw[style=dashed] (-1,0) arc (180:0:1);
\draw[style=dashed] (\d-\r,0) arc (180:0:\r);
\draw[style=dashed] (\c+\r,0) arc (0:180:\r);
\draw[style=thick] (-2.7,0) -- (2.7,0);
\foreach \x/\y in {\c+\r/$\frac{\lambda+1+2\sqrt{\lambda}}{\lambda-1}$,
\c/$\frac{\lambda+1}{\lambda-1}$,
-1/$-1$,
0/$0$,
1/$1$,
\d/$\frac{\lambda+1}{1-\lambda}$,
\d-\r/$\frac{\lambda+1+2\sqrt{\lambda}}{1-\lambda}$}{
\draw (\x,.05) -- (\x,-.05) node [below] {\y};
}
\coordinate [label=below:$\color{gray}\fund$] (F) at (0,2);
\coordinate [label=below:$\iso{s}$] (S) at (0,.95);
\coordinate [label=below:$\iso{h_\lambda^{-1}}$] (H1) at (\c,\r-.05);
\coordinate [label=below:$\iso{h_\lambda}$] (H) at (\d,\r-.05);
\end{tikzpicture}
\caption[exFund]{A Ford fundamental domain for~$\left<h_\lambda,s\right>$.
Since~$h_\lambda$ fixes~$1$ and~$-1$, we have~$\gamma(\R)=\iso{s}$ for each~$\gamma\in\alpha(h_\lambda)$.
Thus, the angle~$\fund$ subtends at the intersection points of the isometric spheres is~$\tfrac{\pi}{2}$ each, which implies that~$\fund$ is a Poincar\'e polygon.}\label{FIG:excludedgroupsfund}
\end{figure}
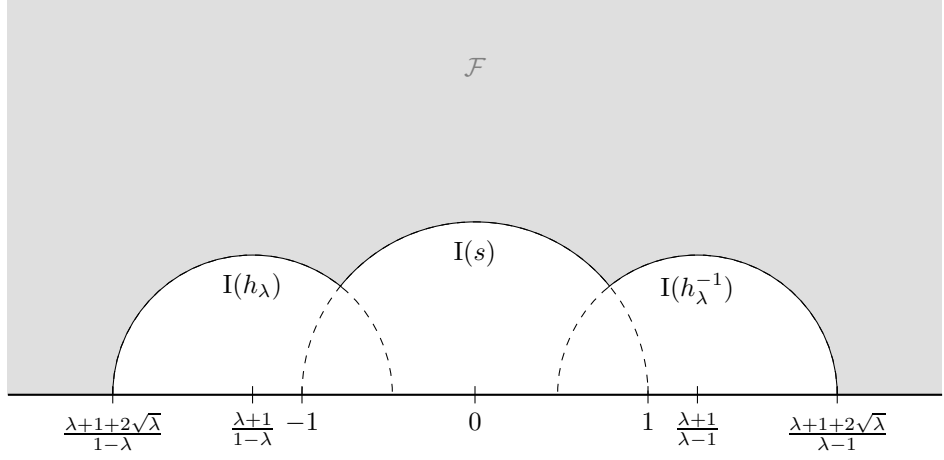

We proceed as described above in order to find a set of branches.
Choose, for instance,
\[
\alpha'\coloneqq\frac{\lambda+1+3\sqrt\lambda}{1-\lambda}
\qquad\text{and}\qquad
\beta'\coloneqq\frac{\lambda+1+3\sqrt\lambda}{\lambda-1}\,.
\]
A set of branches as constructed by the cusp expansion algorithm is indicated in Figure~\ref{FIG:excludedgroupsSoB}.
And from Figure~\ref{FIG:excludedgroupsSoB} it already becomes apparent that the axis of~$h_\lambda$ intersects~$\Cs{\wund,4}$.
Since~each element of~$\alpha(h_\lambda)$ represents a periodic geodesic on the associated orbisurface it follows that~$4\in \neindex$.
In fact, in this example we find~$\neindex=\{4\}$.
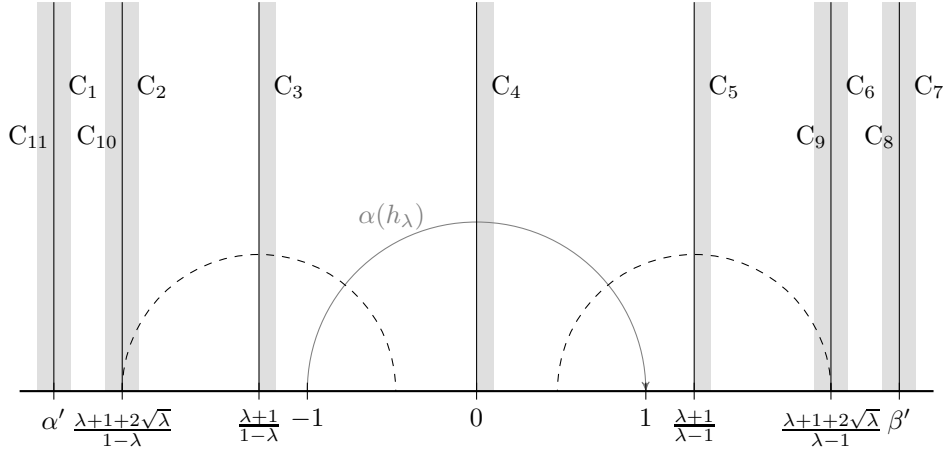
\begin{figure}[h]
\begin{tikzpicture}[scale=2.25]
\tikzmath{\q=1/8;
				  \r=2*sqrt(\q)/abs(\q-1);
				  \c=(\q+1)/(1-\q);
				  \d=(\q+1)/(\q-1);
				  \a=\d-3*\r/2;
				  \b=\c+3*\r/2;
				  }
\foreach \x/\y/\z in {\a/$\Cs{1}$/$\Cs{11}$,
\d-\r/$\Cs{2}$/$\Cs{10}$,
\c+\r/$\Cs{6}$/$\Cs{9}$,
\b/$\Cs{7}$/$\Cs{8}$
}{
\fill[color=lightgray!50] (\x-.1,0) -- (\x+.1,0) -- (\x+.1,2.3) -- (\x-.1,2.3);
\coordinate [label=right:\y] (C\x) at (\x+.03,1.8);
\coordinate [label=left:\z] (C\x) at (\x+.03,1.5);
}
\coordinate [label=above:$\color{gray}\alpha(h_\lambda)$] (AHL) at (-.5,.9);
\foreach \x/\y in {\d/$\Cs{3}$,
0/$\Cs{4}$,
\c/$\Cs{5}$
}{
\fill[color=lightgray!50] (\x,0) -- (\x+.1,0) -- (\x+.1,2.3) -- (\x,2.3);
\coordinate [label=right:\y] (C\x) at (\x+.03,1.8);
}
\draw[<-,color=gray] (1,0) arc (0:180:1);
\draw[style=dashed] (\d-\r,0) arc (180:0:\r);
\draw[style=dashed] (\c+\r,0) arc (0:180:\r);
\draw[style=thick] (-2.7,0) -- (2.7,0);
\foreach \x/\y in {\a/$\alpha'$,
\c+\r/$\frac{\lambda+1+2\sqrt{\lambda}}{\lambda-1}$,
\c/$\frac{\lambda+1}{\lambda-1}$,
-1/$-1$,
0/$0$,
1/$1$,
\d/$\frac{\lambda+1}{1-\lambda}$,
\d-\r/$\frac{\lambda+1+2\sqrt{\lambda}}{1-\lambda}$,
\b/$\beta'$}{
\draw (\x,.05) -- (\x,-.05) node [below] {\y};
}
\foreach \x in {\a,\d-\r,\d,0,\c,\c+\r,\b}{
\draw (\x,0) -- (\x,2.3);
}
\end{tikzpicture}
\caption[exSoB]{A set of branches obtained by application of the cusp expansion algorithm.
The gray stripes indicate that the respective set~$\Cs{j}=\Cs{\wund,j}$ consists of unit tangent vectors based on the adjacent vertical geodesic and pointing into the indicated half-space.
The subscript~$\wund$ is omitted in favor of readability.}\label{FIG:excludedgroupsSoB}
\end{figure}
\end{example}

We further require the following preparations.

\begin{lemma}\label{LEM:vertinGamma}
Let~$j\in A$ and~$g\in\Gamma_\wund$ be such that~$g\act\base{\Cs{\wund,j}}$ is vertical and contained in~$\pr^{-1}((\alpha',\beta'))$.
Then~$g\in\Gamma$.
\end{lemma}

\begin{proof}
Assume~$g\ne\id$, for otherwise there is nothing to show.
Since~$\base{\Cs{\wund,j}}$ and~$g\act\base{\Cs{\wund,j}}$ are both vertical, for~$y\in\pr(g\act\base{\Cs{\wund,j}})$ either
\begin{align*}
y&=g\act\infty\qquad\text{and}\qquad\infty=g\act x_k\,,\qquad\text{or}\\
y&=g\act x_k\qquad\text{and}\qquad g\in\Gamma_{\wund,\infty}\,.
\end{align*}
Because of~\eqref{EQN:SoBinfund} and~$g\ne\id$ the latter case implies that~$\{y,x_k\}=\{\alpha',\beta'\}$, which contradicts the choice of~$j$ and~$g$.
Hence, the former must hold, which implies in particular that~$y$ equals the center of~$\iso{g^{-1}}$ and~$x_k$ equals the center of~$\iso{g}$ (see~\eqref{ITEM:iso:centers}).
Because of that, Lemma~\ref{LEM:centerofREL} implies
\[
\{\iso{g},\iso{g^{-1}}\}\subseteq\REL{\Gamma_\wund}\,.
\]
Combining this with~\eqref{EQN:SoBinfund} and Lemma~\ref{LEM:RelRell} yields
\[
\{\iso{g},\iso{g^{-1}}\}\subseteq\RELL_{\wund}=\REL{\Gamma}\,.
\]
Because of this~\eqref{ITEM:iso:biject} yields a unique~$h\in\Gamma_{\RELL}$ such that~$\iso{h}=\iso{g}$.
By the first statement of~\eqref{ITEM:iso:gener} this implies~$g=t_{\lambda}^nh$ with some~$n\in\Z$.
Now the second statement of~\eqref{ITEM:iso:gener} yields
\[
\iso{h^{-1}}=\iso{g^{-1}t_{\lambda}^n}=t_\lambda^{-n}\act\iso{g^{-1}}\,.
\]
But because of~\eqref{ITEM:iso:relpart} we also have~$\iso{h^{-1}}\in\RELL_{\wund}$, which implies in particular that~$\iso{h^{-1}}\in\pr^{-1}([\alpha',\beta'])$.
Therefore,~$n=0$, which shows~$g\in\Gamma$.
\end{proof}

Recall that~$\alpha,\beta$ have been chosen optimal for
\[
\partial\kund\subseteq\pr^{-1}([\alpha,\beta]_\R)\,.
\]
This is equivalent to the existence of~$\I_1,\I_2\in\REL{\Gamma}$ such that
\[
\alpha\in\geo\I_1\qquad\text{and}\qquad\beta\in\geo\I_2\,.
\]
The following result studies the further boundary structure of~$\kund$.

\begin{lemma}\label{LEM:thirdsphere}
Let~$\I_1,\I_2\in\REL{\Gamma}$ be such that~$\alpha\in\geo\I_1$ and~$\beta\in\geo\I_2$.
Further assume that~$\I_1=\iso{g}$ and $\I_2=\iso{g^{-1}}$ for some~$g\in\Gamma$.
Then either~$\Gamma=\left<g\right>$, or there exists~$\I_3\in\REL{\Gamma}$,~$\I_3\notin\{\I_1,\I_2\}$, such that~$\summit{\I_3}\in\partial\kund$, where~$\summit{\I_3}$ denotes the summit of~$\I_3$ from~\eqref{EQDEF:summit}.
In the latter case we further have
\begin{equation}\label{EQN:Resummits}
\Rea\summit{\I_1}<\Rea\summit{\I_3}<\Rea\summit{\I_2}\,.
\end{equation}
\end{lemma}

\begin{proof}
Assume~$\Gamma$ to be non-cyclic.
Let~$\delta$ be a curve in~$\C$ that traces out the boundary of~$\kund$ in~$\overline{\H}^{\geo}$ between~$\alpha$ and~$\beta$, i.e., let~$a,b\in\R,a<b,$ and let
\[
\delta:I\coloneqq[a,b]_\R\too\C
\]
be a continuous map such that~$\delta(a)=\alpha$,~$\delta(b)=\beta$, and~$\delta(x)\in\partial_\geo\kund$ for every~$x\in(a,b)_\R$.
In particular, because of the boundary structure of~$\kund$,~$\delta$ can be chosen to be injective, for instance by imposing~$\delta$ to be piece-wise parameterized by arc length (w.r.t. the Euclidian metric in~$\C$).
Then~$\delta(I)$ is piece-wise given by either intervals in~$\R$ or geodesic segments in~$\H$.
The function
\[
f:I\ni x\mapstoo\Ima\delta(x)\in\R
\]
is then continuous as well and for every strict local maximum\footnote{by~\emph{strict local maximum} we mean a point~$x_0$ for which there exists an~$\eps>0$ such that~$f(x_0)>f(x)$ for all~$x\in\bigl((x_0-\eps,x_0+\eps)_\R\setminus\{x_0\}\bigr)\cap I$.}~$x_0$ of~$f$ the point~$\delta(x_0)$ coincides with the summit of some relevant isometric sphere.
Furthermore, all non-strict local maxima of~$f$ necessarily are likewise zeros of it.

We start by considering the case~$\I_1\subseteq\partial\kund$.
The combination of~\eqref{ITEM:iso:height} with~\eqref{ITEM:iso:radii} and~\eqref{ITEM:iso:relpart} shows that~$\I_2\subseteq\partial\kund$ as well.
Since~$\Gamma$ is non-cyclic,~$\REL{\Gamma}$ consists of further spheres besides~$\I_1$ and~$\I_2$.
The following argument shows that for at least one of those additional spheres its summit must be contained in~$\partial\kund$.
By construction we find~$x_1,x_2\in(a,b)_\R$ such that
\[
\geo\I_1\setminus\{\alpha\}=\{\delta(x_1)\}\qquad\text{and}\qquad\geo\I_2\setminus\{\beta\}=\{\delta(x_2)\}\,.
\]
Define~$D$ to be the (Euclidean) line segment connecting~$\delta(x_1)$ to~$\delta(x_2)$, i.e.,
\begin{equation}\label{EQN:Delta}
D\coloneqq\defset{(1-t)\delta(x_1)+t\delta(x_2)}{t\in(0,1)}\,.
\end{equation}
Then~$\overline{D}=[\delta(x_1),\delta(x_2)]_\R\subseteq\R$ and, by assumption, the strip~$\pr^{-1}(D)$ contains elements of~$\REL{\Gamma}$ as subsets.
Thus, there exists a point~$x_3\in(x_1,x_2)_\R$ such that~$f(x_3)>0$.
Since~$f$ is continuous, it assumes its maximum in the compact interval~$[x_1,x_2]_\R$.
By the above this implies the existence of~$\I_3\in\REL{\Gamma}$ such that
\[
\summit{\I_3}\in\partial\kund\cap\pr^{-1}(D)\,.
\]
This yields the assertion in the case~$\I_1\subseteq\partial\kund$.

We proceed by assuming that~$\I_1\nsubseteq\partial\kund$.
In this case we find~$x_1,x_2\in(a,b)_\R$ such that~$\delta(x_1)\in\I_1$,~$\delta(x_2)\in\I_2$, but
\[
\delta((x_1,x_1+\eps)_\R)\cap\I_1=\delta((x_2-\eps,x_2)_\R)\cap\I_2=\varnothing\,.
\]
Since~$\Gamma$ is assumed to be non-cyclic we have~$x_1\ne x_2$.
By defining~$D$ again as in~\eqref{EQN:Delta} this time we obtain a horizontal line segment in the upper half-plane (see~\eqref{ITEM:iso:height}).
We obtain another continuous path in~$\C$ by connecting the segments~$\delta((a,x_1)_\R)$,~$D$, and~$\delta((x_2,b)_\R)$.
The angles this path assumes at the points~$\delta(x_1)$ and~$\delta(x_2)$ (measured above the curve) are equal and we denote them by~$\vartheta$.
Necessarily,
\begin{equation}\label{EQN:Deltaangle}
\frac{\pi}{2}<\vartheta<\pi\,.
\end{equation}
By assumption~$\I_1\nsubseteq\partial\kund$ there exists~$\I_3\in\REL{\Gamma}\setminus\{\I_1,\I_2\}$ such that
\[
\delta([x_1,x_1+\eps)_\R)\in\I_3\quad\text{for }\eps>0\text{ sufficiently small.}
\]
The points~$v_1\coloneqq\delta(x_1)$ and~$v_2\coloneqq\delta(x_2)$ are finite vertices of~$\kund$.
Denote the angles~$\kund$ subtends at~$v_1$ and~$v_2$ as before by~$\theta_1$ and~$\theta_2$ respectively.
Clearly,
\begin{equation}\label{EQN:anglestheta}
\theta_j<\pi\qquad\text{for }j\in\{1,2\}\,.
\end{equation}
If~$\theta_1<\vartheta$, then~$\delta((x_1,x_1+\eps)_\R)$ lies above~$D$.
Thus, we can proceed as above to find a strict local maximum of~$f$ in the interval~$[x_1,x_2]_\R$ and thereby a summit of some relevant isometric sphere contained in~$\partial\kund\cap\pr^{-1}(\pr(D))$.
If~$\theta_1=\vartheta$, then~$D$ is contained in a line tangent to~$\I_3$.
This means~$\delta(x_1)$ is the summit of~$\I_3$ and thus fulfills the assertion, since it clearly is contained in~$\partial\kund$.
The same arguments apply if the angle~$\theta_2$ falls below or equals~$\vartheta$.
This leaves the case of both these angles exceeding~$\vartheta$, which, because of~\eqref{EQN:Deltaangle} and~\eqref{EQN:anglestheta}, implies
\[
\pi<\theta_1+\theta_2<2\pi\,.
\]
This means that the equality in Corollary~\ref{COR:anglesum} can only be fulfilled for~$\omega=1$ and the cycle~$C(v_1)$ consisting of more vertices besides~$v_1,v_2$.
Let~$v_3\in C(v_1)\setminus\{v_1,v_2\}$ and denote by~$\theta_3$ the angle~$\kund$ subtends at~$v_3$.
We then have
\begin{equation}\label{EQN:theta3}
\theta_3\leq 2\pi-(\theta_1+\theta_2)<\pi\,.
\end{equation}
We further find~$x_3\in(x_1,x_2)_\R$ such that~$\delta(x_3)=v_3$.
Since~$\kund$ is a Ford fundamental domain and~$v_1\in\partial\kund$, Lemma~\ref{LEM:vertexcycles} yields~$v_3\in D$.
But now, because of~\eqref{EQN:theta3}, we find a small~$\eps>0$ such that at least one of the segments~$\delta((x_3-\eps,x_3)_\R)$ or~$\delta((x_3,x_3+\eps)_\R)$ lies above~$D$.
Hence, by the same argument as before, we find a strict local maximum of~$f$, either in~$[x_1,x_3]_\R$, or in~$[x_2,x_3]_\R$.
This yields the assertion in the case~$\I_1\nsubseteq\partial\kund$.

In order to show validity of~\eqref{EQN:Resummits}, assume it is not the case.
Because of symmetry it suffices to consider the case~$\Rea\summit{\I_3}\leq\Rea\summit{\I_1}$.
Since~$\summit{\I_3}\in\partial\kund$, we have~$\summit{\I_3}\notin\mathrm{int\,}\I_1$.
But, because of convexity, this implies that~$\geo\I_3\cap(-\infty,\alpha)_\R\ne\varnothing$, contradicting the choice of~$\I_1$.
This finishes the proof.
\end{proof}

\begin{proof}[Proof of Proposition~\ref{LEM:Anotempty}]
Because of Example~\ref{EX:excludedgroups} it suffices to consider groups~$\Gamma$ non-conjugate to a group generated by~$z\mapsto-\tfrac{1}{z}$,~$z\mapsto\lambda z$, for any~$\lambda>1$.

The set~$\REL{\Gamma}$ is finite and contains spheres~$\I_1,\I_2$ such that
\[
\alpha\in\geo\I_1\qquad\text{and}\qquad\beta\in\geo\I_2\,.
\]
Since~$\Gamma$ is assumed to contain hyperbolic elements,~$\REL{\Gamma}$ is not a singleton and thus,~$\I_1\ne\I_2$.
Because of~\eqref{ITEM:iso:biject} there exist uniquely determined~$g_1,g_2\in\Gamma$,~$g_1\ne g_2$, such that~$\I_1=\iso{g_1}$ and~$\I_2=\iso{g_2}$.
Consider the segments of~$\I_1,\I_2$ contributing to the boundary of~$\wund$, that are the sets
\[
b_j\coloneqq\partial\wund\cap\I_j\,,
\]
for~$j=1,2$.
If~$b_1\cap b_2\ne\varnothing$, then~$\REL{\Gamma}=\{\I_1,\I_2\}$.
This is only possible in the case ~$g_1=g_2^{-1}$ and~$g_1$ elliptic.
But this means that~$\Gamma$ is cyclic, generated by an elliptic transformation, and therefore void of hyperbolic elements. Since this contradicts the assumption, we conclude
\begin{align}\label{EQN:relpartsdisjoint}
b_1\cap b_2=\varnothing\,.
\end{align}

From here on we distinguish the cases~$g_1=g_2^{-1}$ and~$g_1\ne g_2^{-1}$, starting with the latter.
Because of~\eqref{ITEM:iso:relpart} the segments~$g_j\act b_j$ are subsets of~$\partial\wund$ as well.
Since~$b_1,b_2$ are geodesic segments with at least one endpoint in~$\partial_{\geo}\H$ and every element of~$\Gamma$ fixes~$\partial_{\geo}\H$, the geodesic segments~$g_1\act b_1, g_2\act b_2$ have one endpoint in~$\partial_{\geo}\H$ as well.
Hence, because of~$g_1\ne g_2^{-1}$ and the choice of~$\I_1,\I_2$, there exists an interval~$I\subseteq(\alpha,\beta)$ representing a funnel of~$\Gamma$ such that
\begin{align}\label{EQN:disjointstrip}
\pr^{-1}(I)\subseteq\wund\quad\text{and}\quad\pr^{-1}(I)\cap\pr^{-1}(\underiso{g_j})=\varnothing\,, 
\end{align}
for~$j=1,2$, with~$\underiso{g_j}$ as in Lemma~\ref{LEM:isolimit}.
Because of~(\ref{ITEM:cex:tessellation}\ref{ITEM:tessellation:structure}) and~\eqref{ITEM:cex:uniqueness} there exist two pairs~$(k_1,h_1),(k_2,h_2)\in A\times\Gamma_\wund$ such that~$h_j\act\base{\Cs{\wund,k_j}}$ is a vertical side of~$\pr^{-1}(I)$ for~$j=1,2$, and~$h_1\act\mathrm{H}_{\pm}(k_1)=h_2\act\mathrm{H}_{\mp}(k_2)$.
Because of Lemma~\ref{LEM:vertinGamma} we have~$h_1,h_2\in\Gamma$.
Because of Lemma~\ref{LEM:isolimit} and the openness of the sets~$\underiso{g_j}$ there exist~$x_1,x_2\in\Lambda(\Gamma)$ as well as~$\eps>0$ such that
\[
(x_j-\eps,x_j+\eps)_\R\subseteq\underiso{g_j}\,,
\]
for~$j=1,2$.
The density of~$E(\Orbi)$ in~$\Lambda(\Gamma)\times\Lambda(\Gamma)$ implies the existence of a geodesic~$\gamma\in\Geo_{\Per,\Gamma}(\H)$ for which we have
\begin{align}\label{EQN:endpointsunderiso}
\gamma(-\infty)\in(x_1-\eps,x_1+\eps)_\R\qquad\text{and}\qquad\gamma(+\infty)\in(x_2-\eps,x_2+\eps)_\R\,.
\end{align}
The combination of~\eqref{EQN:disjointstrip} and~\eqref{EQN:endpointsunderiso} implies
\[
(\gamma(+\infty),\gamma(-\infty))\in h_j\act I_{k_j}\times h_j\act J_{k_j}
\]
for one~$j\in\{1,2\}$.
Combining this with~\eqref{ITEM:cex:vectors} implies~$k_j\in \neindex$.

Now assume that~$g\coloneqq g_1=g_2^{-1}$.
If~$\Gamma$ is cyclic, meaning~$\Gamma=\left<g\right>$, then there again exists an interval~$I\subseteq(\alpha,\beta)_\R$ which fulfills~\eqref{EQN:disjointstrip}.
Using~(\ref{ITEM:cex:tessellation}\ref{ITEM:tessellation:structure}) and Lemma~\ref{LEM:vertinGamma} we find~$(k,h)\in A\times\Gamma$ such that every representative of~$\alpha(g)$ intersects~$h\act\Cs{\wund,k}$, which implies~$k\in A'$.

Hence, we may assume that~$\Gamma$ is non-cyclic.
Let~$\xi_1,\xi_2\in\overline{\H}^{\geo}$ be such that
\[
\overline{b_1}^{\geo}=[\alpha,\xi_1]_{\H}\qquad \text{and}\qquad\overline{b_2}^{\geo}=[\xi_2,\beta]_{\H}\,.
\]
Because of \eqref{ITEM:iso:relpart} we have~$g_1\act b_1=b_2$, wherefore~\eqref{ITEM:iso:height} implies
\begin{align}\label{EQN:ImaIma}
\Ima\xi_1=\Ima\xi_2\,.
\end{align}
Because of~\eqref{EQN:relpartsdisjoint} the boundary of~$\wund$ consists of further segments besides~$b_1,b_2$, and
because of Lemma~\ref{LEM:thirdsphere} at least one of these segments, say~$b_3$ induced by the relevant isometric sphere~$\I_3$, contains the summit~$s=\summit{\I_3}$.
Application of~\eqref{ITEM:cex:summits} now yields a pair~$(k,h)\in A\times\Gamma_\wund$ such that~$h\act\base{\Cs{\wund,k}}$ is vertical with~$s\in h\act\base{\Cs{\wund,k}}$.
Again, Lemma~\ref{LEM:vertinGamma} yields~$h\in\Gamma$.

By following the structure of the argument above and taking the~$\Gamma$-invariance of~$\Lambda(\Gamma)$ into account, we see that it suffices to show that~$\mittelp{\I_3}=\pr(s)$ separates at least some points in~$\underiso{g}\cap\Lambda(\Gamma)$ from at least some points in~$\underiso{g^{-1}}\cap\Lambda(\Gamma)$, or in other words,
\begin{equation}\label{EQ:leftrightnonempty}
\underiso{g}\cap\Lambda(\Gamma)\cap h\act J_k\ne\varnothing\qquad\text{and}\qquad\underiso{g^{-1}}\cap\Lambda(\Gamma)\cap h\act I_k\ne\varnothing\,.
\end{equation}
In order to see this we distinguish several cases, starting with the assumption that~$g$ is hyperbolic.
Then~$\iso{g}\cap\iso{g^{-1}}=\varnothing$, and thus,~$\underiso{g}\cap\underiso{g^{-1}}=\varnothing$.
If~$\mittelp{\I_3}\notin\underiso{g}\cup\underiso{g^{-1}}$, then~\eqref{EQ:leftrightnonempty} is immediately clear.
Thus, assume that this is not the case.
Because of symmetry it entails no loss of generality to assume that~$\mittelp{\I_3}\in\underiso{g}$.
Then~$\underiso{g^{-1}}\subseteq h\act I_k$ by construction, and hence Lemma~\ref{LEM:isolimit} implies that~$\Lambda(\Gamma)\nsubseteq h\act J_k$.
Suppose for contradiction that~$\Lambda(\Gamma)\subseteq h\act I_k$.
Denote by~$g_3$ the generator of~$\I_3$, which is unique by~\eqref{ITEM:iso:biject}.
By construction,~$g_3\notin\{g,g^{-1}\}$.
The transformation~$g_3$ cannot be an involution, for then~$g_3h\act I_k= h\act J_k$.
And since $\Gamma$-action preserves~$\Lambda(\Gamma)$, we obtain a contradiction to the assumption.
Therefore,~$\mittelp{g_3}\ne\mittelp{g_3^{-1}}$, and we show that
\begin{equation}\label{EQ:cg31loc}
\mittelp{g_3^{-1}}\in\bigl(\mittelp{g_1},\mittelp{g_2}-r\bigr)\,,
\end{equation}
where~$r\coloneqq\radius{g_1}=\radius{g_2}$.
To that end we first show that
\begin{equation}\label{EQ:cg3cg31loc}
\{\mittelp{g_3},\mittelp{g_3^{-1}}\}\subseteq\bigl(\mittelp{g_1},\mittelp{g_2}\bigr)\,.
\end{equation}
Let~$x\in\{\mittelp{g_3},\mittelp{g_3^{-1}}\}$.
Since~$r'\coloneqq\radius{g_3}=\radius{g_3^{-1}}$, we then have
\[
x+\i r'\in\{\summit{g_3},\summit{g_3^{-1}}\}\,.
\]
Since~$g_3\act\summit{g_3}=\summit{g_3^{-1}}$, we obtain from~$\summit{\I_3}\in b_3$, and~\eqref{ITEM:iso:relpart} that
\[
\{\summit{g_3},\summit{g_3^{-1}}\}\subseteq\partial\wund\,.
\]
Hence, in particular neither summit is contained in~$\intiso{g}\cup\intiso{g^{-1}}$.
Since~$\I_3\notin\{\I_1,\I_2\}$, it follows from~\eqref{ITEM:iso:biject} that neither summit is contained in~$\I_1\cup\I_2$ either.
But then, for~$x\leq\mittelp{g}$ we find
\[
x-r'<\mittelp{g}-r=\alpha\,,
\]
while for~$x\geq\mittelp{g^{-1}}$ we find
\[
x+r'>\mittelp{g^{-1}}+r=\beta\,,
\]
Thus, either case entails a contradiction to the choice of~$\I_1,\I_2$.
This yields~\eqref{EQ:cg3cg31loc}.
By the assumption~$\mittelp{g_3}\in\underiso{g}$ the geodesic arc~$\base{\Cs{\wund,k}}=(\mittelp{g_3},\infty)_\H$ intersects~$\I_1=\iso{g}$ in exactly one point in~$\H$, say~$\xi_3$.
Therefore, the geodesic arc~$g\act\base{\Cs{\wund,k}}$ intersects~$\I_2=\iso{g^{-1}}=g\act\I_1$ exactly in~$g\act\xi_3$.
Since we have
\[
\pr(\xi_3)\in(\mittelp{g},\mittelp{g}+r)\,,
\]
by~\eqref{EQ:cg3cg31loc} and~$g\act(\mittelp{g}+r)=\mittelp{g^{-1}}-r$ we find~$\pr(g\act\xi_3)\in(\mittelp{g^{-1}}-r,\mittelp{g^{-1}})$.
By combining this with~$g\act\infty=\mittelp{g^{-1}}$, we conclude that~$g\act\base{\Cs{\wund,k}}$ is non-vertical and
\[
\bigl(\mittelp{g^{-1}}-r,\mittelp{g^{-1}}\bigr)\varsubsetneq\pr\bigl(g\act\base{\Cs{\wund,k}}\bigr)\,.
\]
Thus, if~$\mittelp{g_3^{-1}}\in(\mittelp{g^{-1}}-r,\mittelp{g^{-1}})$, then, because of
\begin{equation}\label{EQ:g3CkWk}
g_3\act\base{\Cs{\wund,k}}=\bigl(g_3\act\mittelp{\I_3},\infty\bigr)_\H=\bigl(\infty,\mittelp{g_3^{-1}}\bigr)_\H\,,
\end{equation}
the geodesic arcs~$g_3\act\base{\Cs{\wund,k}}$ and~$g_1\act\base{\Cs{\wund,k}}$ intersect each other without coinciding.
Since~$\Gamma\subseteq\Gamma_\wund$ and~$\BrS_\wund$ is a set of branches for the geodesic flow on~$\Orbi_\wund$, this yields a contradiction by violation of~\eqref{BP:disjointunion}.
Because of~\eqref{EQ:cg3cg31loc} this yields~\eqref{EQ:cg31loc}.
Now, by combination of~\eqref{EQ:cg31loc} with
\[
g_3h\act I_k=g_3\act\bigl(\mittelp{g_3},+\infty\bigr)=\bigl(-\infty,\mittelp{g_3^{-1}}\bigr)\,,
\]
the assumptions~$\underiso{g^{-1}}\subseteq h\act I_k$,~$\mittelp{g_3}\in\underiso{g}$, and~$\underiso{g}\cap\underiso{g^{-1}}=\varnothing$, and the identity~\eqref{EQ:g3CkWk}, we infer
\[
\underiso{g^{-1}}\subseteq g_3h\act J_k\,.
\]
Hence, the same argument which showed that~$g_3$ cannot be an involution again yields a contradiction.
Hence,~$\Lambda(\Gamma)\subseteq h\act I_k$ cannot hold true, which in turn implies~\eqref{EQ:leftrightnonempty}.
This yields the assertion in the case that~$g$ is hyperbolic.

Now assume that~$g$ is elliptic of order~$\sigma$.
Since~$\iso{g}\ne\iso{g^{-1}}$ by assumption, we have~$\sigma\geq3$.
Hence, the angle between~$\iso{g}$ and~$\iso{g^{-1}}$ at the fixed point~$\fixp{}{g}$ exceeds~$\tfrac{2\pi}{3}$ (measured above the spheres).
Since~$\wund$ is geometrically finite, we may enumerate its sides as~$a_1,\dots,a_m$ from left to right, i.e., such that
\[
b_1=a_1\,,\quad a_m=b_2\,,\quad\text{and}\quad a_i\cap a_{i+1}\ne\varnothing\,,
\]
for~$i=1,\dots,m-1$.
Analogously, we may enumerate the elements of~$V_\wund$, the finite vertices of~$\wund$, by~$v_1,\dots,v_{m-1}$ such that~$\{v_i\}=a_i\cap a_{i+1}$, for all~$i$.
Finally, denote the angle~$\wund$ subtends at~$v_i$ by~$\theta_i$.
Since we have~$\iso{g}\cap\iso{g^{-1}}\ne\varnothing$ and
\[
\pr(a_i)\subseteq\underiso{g}\cup\underiso{g^{-1}}
\]
for every~$i\in\{2,\dots,m-1\}$,
we conclude that~$\I_{a_i}\cap\I_\iota\ne\varnothing$, for~$\I_{a_i}$ the relevant isometric sphere inducing the segment~$a_i$ and~$\iota\in\{1,2\}$.
Since~$v_i\in\overline{\extiso{g}\cap\extiso{g^{-1}}}$ for every~$i$, this implies
\begin{align}\label{EQN:estitheta}
\frac{2\pi}{3}<\theta_i<\pi
\end{align}
for all~$i$.
Consider the vertex cycle~$C(v_1)=\{v_{i_1},\dots,v_{i_\ell}\}$ with~$v_{i_1}=v_1$.
Then
\[
g\act v_1=v_{m-1}\in C(v_1)
\]
and hence~$\ell>1$.
Because of Corollary~\ref{COR:anglesum} there exists~$\omega\in\N$ such that
\[
\frac{2\pi}{\omega}=\theta(C(v_1))=\sum_{\kappa=1}^\ell\theta_{i_\kappa}\stackrel{\eqref{EQN:estitheta}}{>}\frac{2\ell\pi}{3}\,,
\]
which implies~$\ell\omega<3$.
Since~$\ell,\omega\in\N$ and~$\ell>1$, this leaves
\[
(\ell,\omega)=(2,1)
\]
As the only possible configuration.
But this implies
\[
C(v_1)=\{v_1,v_{m-1}\}\qquad\text{and}\qquad\theta_1+\theta_{m-1}=2\pi\,,
\]
which means that at least one of the two angles equals or exceeds~$\pi$, in violation of the second relation in~\eqref{EQN:estitheta}.
Hence, this final case is contradictory and the proof is finished.
\end{proof}

In the proof of Proposition~\ref{LEM:Anotempty}, for any given constellation, we identified a hyperbolic transformation~$g\in\Gamma$ with fixed points~$\fixp{+}{g}$ and~$\fixp{-}{g}$ sufficiently far apart such that there exists~$k\in A$ and~$h\in\Gamma_\wund$ for which~$h\act\Cs{\wund,k}$ is intersected by~$\alpha(g)$.
This then yielded~$k\in\neindex$, and because~$(k,h)$ could be chosen such that~$h\act\base{\Cs{\wund,k}}$ is vertical and, necessarily,
\[
\pr\bigl(h\act\base{\Cs{\wund,k}}\bigr)\in(a',b')\,,
\]
Lemma~\ref{LEM:vertinGamma} yields~$h\in\Gamma$.
The same argumentation also applies for~$g^{-1}$, with the roles of~$g_1$ and~$g_2$ in the proof of Proposition~\ref{LEM:Anotempty} interchanged.
Hence, we obtain a second branch copy~$h'\act\Cs{\wund,k'}$,~$(k',h')\in\neindex\times\Gamma$, pointing in the opposite direction of~$h\act\Cs{\wund,k}$, i.e.,
\[
h\act I_{\wund,k}=(h\act x_k,+\infty)\qquad\text{and}\qquad h'\act I_{\wund,k'}=(-\infty,h'\act x_{k'})\,.
\]
Therefore, the union~$h\act I_{\wund,k}\cup h'\act I_{\wund,k'}$ covers~$\R$ except, perhaps, for a bounded interval.
Since
\[
\bigl(\fixp{+}{g},\fixp{-}{g}\bigr)\in h'\act J_{\wund,k'}\times h'\act I_{\wund,k'}\,,
\]
iterated application of~$g$ contracts~$h'\act\base{\Cs{\wund,k'}}$ towards~$\fixp{+}{g}$.
Or in other words, there exists~$n\in\N$ such that
\[
g^nh'\act\base{\Cs{\wund,k'}}\subseteq h\act\Plussp{k}\,,
\]
with~$\Plussp{k}$ the half-space from~\eqref{BP:pointintohalfspaces}.
This means
\[
h\act J_{\wund,k}\subseteq g^nh'\act I_{\wund,k'}\qquad\text{and}\qquad g^nh'\act J_{\wund,k'}\subseteq h\act I_{\wund,k}\,, 
\]
which in turn yields the following result.

\begin{cor}\label{COR:backsidegamma}
There exist~$j,k\in A'$ and~$g,h\in\Gamma$ such that
\[
\R=g\act I_{\wund,j}\cup h\act I_{\wund,k}\,.
\]
\end{cor}

Recall the transition sets~$\Trans{\wund}{.}{.}$ associated to~$\BrS_\wund$ by~\eqref{BP:intervaldecomp}.
The following lemma is key.

\begin{lemma}\label{LEM:transGamma}
For every choice of~$j,k\in \neindex$ we have~$\Trans{\wund}{j}{k}\subseteq\Gamma$.
\end{lemma}

\begin{proof}
Fix~$j\in \neindex$ and let~$k\in \neindex$ be such that~$\Trans{\wund}{j}{k}\ne\varnothing$.
For~$g\in\Trans{\wund}{j}{k}$ consider
\[
b_{(k,g)}\coloneqq g\act\base{\Cs{\wund,k}}=(g\act x_k,g\act\infty)_{\H}\,.
\]
This is a complete geodesic segment contained in the half-space~$\Plussp{j}$.
A priori, it might be vertical or non-vertical.
Since~$x_k\in\{a',b'\}$ implies that one of the sets~$\Iset{\wund,k},\Jset{\wund,k}$ is empty and thus~$k\notin\neindex$ in violation of the assumption, the assertion in the vertical case has already been shown in Lemma~\ref{LEM:vertinGamma}.

Therefore, assume that~$b_{(k,g)}$ is non-vertical.
From~\eqref{ITEM:cex:tessellation} we see that there exists~$B\in\mathscr{B}$ such that~$b_{(k,g)}$ and~$b_{(j,\id)}$ are both sides of~$B$.
If~$B$ were a strip, meaning of the form~$\pr^{-1}(L)$ for some interval~$L\subseteq\R$ as described in~(\ref{ITEM:cex:tessellation}\ref{ITEM:tessellation:strip}), then~$b_{(k,g)}$ would be vertical and we would be in the above case.
Thus,~$B$ is a hyperbolic polyhedron with~$b_{(j,\id)}$ being one of its two vertical sides.
Assume first that~$B$ is a hyperbolic triangle.
Then either
\[
x_j=g\act\infty\,,\qquad\text{or}\qquad x_j=g\act x_k\,.
\]
In the former case, application of Lemma~\ref{LEM:centerofREL} yields~$\iso{g^{-1}}\in\REL{\Gamma_{\wund}}$, and we can proceed as above to conclude~$g\in\Gamma$.
In the latter case, consider the other vertical side of~$B$.
Because of~(\ref{ITEM:cex:tessellation}\ref{ITEM:tessellation:structure}) there exists a pair~$(j',h)\in A\times\Gamma_\wund$ such that this side is given by~$b_{(j',h)}$ and we further have~$h^{-1}g\in\Trans{\wund}{j'}{k}$.
This implies in particular that~$j'\in \neindex$.
We now find either
\[
g\act\infty=h\act x_{j'}\,,\qquad\text{or}\qquad g\act\infty=h\act\infty\,.
\]
The former case implies~$h\in\Gamma_{\wund,\infty}$, which, with the same argument as above, can only hold true if~$h=\id$.
Hence,~$g\act\infty=x_{j'}$ and we argue as before with~$j'$ in place of~$j$.
Because of Lemma~\ref{LEM:isoconcentric} the latter case implies~$g=h$.
Since we further find~$h\act x_{j'}=\infty$, we can argue as before to conclude~$g\in\Gamma$.

Now assume that~$B$ is not a hyperbolic triangle.
Since~$B$ is also assumed to not be a strip, from~(\ref{ITEM:cex:tessellation}\ref{ITEM:tessellation:triangle}) we obtain that each side of~$B$ is of the form~$(h^{\ell}\act\infty,h^{\ell+1}\act\infty)_{\H}$ for some elliptic~$h\in\Gamma_{\wund}$ and~$\ell\in\{0,\dots,\ord(h)-1\}$. In particular, we can choose~$h$ so that
\[
b_{(k,g)}=(h^{\ell'}\act\infty,h^{\ell'+1}\act\infty)_{\H}\qquad\text{and}\qquad h^{-\ell'}g\act\infty=\infty
\]
for some~$\ell'\in\{1,\dots,\ord(h)-2\}$.
Then~$
b_{(j,\id)}=(h^{\iota}\act\infty,\infty)_{\H}
$
for some~$\iota\in\{\pm1\}$, meaning that~$h^{-\iota}\act x_j=\infty$.
As before, this implies~$h\in\Gamma$.
Applying again Lemma~\ref{LEM:isoconcentric} we find~$g=h^{\ell}$ for some~$\ell\in\{0,\dots,\ord(h)-1\}$.
This finishes the proof.
\end{proof}

We are now ready to prove our main result, identifying a set of branches for the geodesic flow on the orbisurface~$\Orbi$ without cusps.
Evidently, the proof makes use of~$\BrS_{\wund}$ being a set of branches for the geodesic flow on~$\Orbi_{\wund}$.
In order to distinguish between the defining properties from Definition~\ref{DEF:setofbranches} in the two different contexts, we denote those fulfilled by~$\BrS_\wund$ with respect to~$\Gamma_\wund$ by~(\ref{BP:closedgeodesicsHtoX}$_{\rwund}$)--(\ref{BP:intervaldecomp}$_{\rwund}$) respectively.

\begin{thm}\label{THM:SoB}
$\BrS_\wund'\coloneqq\defset{\Cs{\wund,j}}{j\in \neindex}$ is a set of branches for the geodesic flow on~$\Orbi$.
\end{thm}

\begin{proof}
From Proposition~\ref{LEM:Anotempty} and~$\neindex\subseteq A$ we see that~$\BrS_\wund'$ is a finite but non-empty set.
The definition of~$\neindex$ combined with~(\ref{BP:closedgeodesicsHtoX}$_{\rwund}$) further assures validity of~\eqref{BP:closedgeodesicsHtoX}.
Let~$j\in \neindex$. Since~$[\infty]_{\Gamma_\wund}$ is the only cusp of~$\Orbi_\wund$,~\eqref{ITEM:cex:setstructure} implies that the point~$x_j$ either equals the center of some relevant isometric sphere, or it is contained in a representative of a funnel of~$\Orbi_\wund$.
Since~$\infty$ is contained in a representative of a funnel of~$\Orbi$, so is every center of an isometric sphere for~$\Gamma$, by virtue of~\eqref{ITEM:iso:centers}.
Therefore,~\eqref{BP:completegeodesics} follows directly from~(\ref{BP:completegeodesics}$_{\rwund}$).
Property~\eqref{BP:pointintohalfspaces} is immediate from~(\ref{BP:pointintohalfspaces}$_{\rwund}$) and property~\eqref{BP:coverlimitset} follows from Corollary~\ref{COR:backsidegamma} and~$\infty$ being an inner point of~$\Omega(\Gamma)$.
The properties~\eqref{BP:allvectors} and~\eqref{BP:disjointunion} follow from~(\ref{BP:allvectors}$_{\rwund}$) and~(\ref{BP:disjointunion}$_{\rwund}$) respectively, by taking Corollary~\ref{COR:limitsetinclude} into account.
Finally,~\eqref{BP:intervaldecomp} follows from~(\ref{BP:intervaldecomp}$_{\rwund}$) and Lemma~\ref{LEM:transGamma}.
\end{proof}

\section{Strict transfer operator approaches}\label{SEC:strictTOA}

Let~$\Orbi$ be a non-compact developable hyperbolic orbisurface with fundamental group~$\Gamma$ fulfilling~\eqref{condA}.
By virtue of either Theorem~\ref{THM:cuspexpSoB} (if~$\Orbi$ has cusps) or Theorem~\ref{THM:SoB} (if it does not) we obtain a set of branches for the geodesic flow on~$\Orbi$.
Theorem~\ref{THM:mainthmPW} now implies that~$\Gamma$ admits a strict transfer operator approach, where the required (slow or fast) transfer operator families are given explicitly by the set of branches (see \cite[Section~3.6 and Section~7]{Pohl_Wab}).
This yields validity of Theorem~\ref{THMA}.

\section*{Appendix}

\subsection*{Proof of Theorem~\ref{THM:cuspexpSoB}}

The finiteness of~$\BrS_{\P}$ follows immediately from~\eqref{ITEM:cex:finiteset}.

Recall the set~$E(\Orbi)$ from~\eqref{EQNDEF:EX} and its density in~$\Lambda(\Gamma)\times\Lambda(\Gamma)$.
Combining this with~\eqref{ITEM:cex:vectors} yields~\eqref{BP:closedgeodesicsHtoX}.

Since~$\infty$ represents a cusp of~$\Orbi$ and is therefore an element of $\wh\R\setminus\wh\R_{\st}$ the property~\eqref{BP:completegeodesics} follows from~\eqref{ITEM:cex:setstructure}.

Property~\eqref{BP:pointintohalfspaces} is the same as the first statement in~\eqref{ITEM:cex:vectors}, while the second statement in~\eqref{ITEM:cex:vectors} yields~\eqref{BP:allvectors}.

Property~\eqref{BP:coverlimitset} follows from~\eqref{ITEM:cex:uniqueness}: Let~$j\in A$ and pick~$(k,g)\in A\times\Gamma$ according to~\eqref{ITEM:cex:uniqueness}.
Then, since~$x_j,\infty\in\widehat\R\setminus\widehat\R_{\st}$,
\[
\wh\R_{\st}\subseteq\wh\R\setminus\{x_j,\infty\}=I_{\P,j}\cup J_{\P,j}=I_{\P,j}\cup g\act I_{\P,k}\,.
\]

Property~\eqref{BP:disjointunion} is a consequence of~\eqref{ITEM:cex:tessellation} and~\eqref{ITEM:cex:uniqueness}: Let~$j,k\in A$ and~$g\in\Gamma$ be such that
\[
\base{\Cs{\P,j}}\cap g\act\base{\Cs{\P,k}}\ne\varnothing\,.
\]
Because of~(\ref{ITEM:cex:tessellation}\ref{ITEM:tessellation:include}) there exist~$B_1,B_2\in\mathscr{B}$ such that~$\base{\Cs{\P,j}}$ and~$g\act\base{\Cs{\P,k}}$ are maximal connected components of~$\partial B_1$ and~$\partial B_2$ respectively.
Because of~\eqref{ITEM:cex:setstructure},~(\ref{ITEM:cex:tessellation}\ref{ITEM:tessellation:structure}), and the exactness of the tessellation provided by~$\mathscr{B}$, this implies
\[
\base{\Cs{\P,j}}=g\act\base{\Cs{\P,k}}\,.
\]
Combining this with~\eqref{ITEM:cex:uniqueness} yields~\eqref{BP:disjointunion}.

Finally, in order to verify~\eqref{BP:intervaldecomp} let~$j\in A$ and let~$\nu\in\Cs{\P,j,\st}$.
Because of~\eqref{ITEM:cex:nextinter} the number~$t_+(\nu)$ is well-defined.
From~\eqref{BP:disjointunion} we infer that the transformation~$g_+(\nu)\in\Gamma$ as well as the index~$k_+(\nu)\in A$ such that
\[
\gamma_\nu'(t_+(\nu))\in g_+(\nu)\act\Cs{\P,k_+(\nu)}
\]
are both uniquely determined.
By construction we have
\[
g_+(\nu)\act\Plusspp{k_+(\nu)}\subseteq\Plusspp{j}\,,
\]
hence,
\begin{align}\label{EQN:Isubset}
g_+(\nu)\act I_{\P,k_+(\nu)}\subseteq I_{\P,j}\,.
\end{align}
For~$k\in A$ we set~$\Cs{\P,j}\vert_k\coloneqq\defset{\nu\in\Cs{\P,j}}{k_+(\nu)=k}$.
Because of the well-definedness of~$k_+(\nu)$ for every~$\nu\in\Cs{\P,j,\st}$ we obtain
\begin{align}\label{EQN:CPCPk}
\Cs{\P,j,\st}=\bigcup_{k\in A}\Cs{\P,j}\vert_k\,,
\end{align}
where the union is clearly disjoint.
We further define
\[
\Trans{}{j}{k}\coloneqq\bigcup_{\nu\in\Cs{\P,j}\!\vert_k}\{g_+(\nu)\}\,.
\]
Then the sets~$\Cs{\P,j}\vert_k$ decompose further as
\begin{align}\label{EQN:CPdecomp}
\Cs{\P,j}\vert_k=\bigcup_{g\in\Trans{}{j}{k}}\defset{\nu}{g_+(\nu)=g}\,.
\end{align}
Again, the union is disjoint.
Because of~\eqref{EQN:Isubset} we further have
\[
J_{\P,j,\st}\subseteq g_+(\nu)\act J_{\P,k_+(\nu),\st}\,.
\]
By combining this with~\eqref{BP:allvectors} and~\eqref{EQN:Isubset} we obtain
\[
\defset{\gamma_\nu(+\infty)}{\nu\in\Cs{\P,j}\vert_k\,,\ g_+(\nu)=g}_{\st}=\defset{\gamma_{g\act\eta}(+\infty)}{\eta\in\Cs{\P,k}}_{\st}\,.
\]
Combination of this with~\eqref{EQN:CPCPk},~\eqref{EQN:CPdecomp}, and~\ref{ITEM:cex:vectors} in turn yields
Combining this with~\eqref{EQN:CPCPk} and the last statement of~\eqref{ITEM:cex:vectors} yields
\begin{align*}
I_{\P,j,\st}&=\defset{\gamma_\nu(+\infty)}{\nu\in\Cs{\P,j}}_{\st}=\bigcup_{k\in A}\defset{\gamma_\nu(+\infty)}{\nu\in\Cs{\P,j}\vert_k}_{\st}\\
&=\bigcup_{k\in A}\bigcup_{g\in\Trans{}{j}{k}}\defset{\gamma_\nu(+\infty)}{\nu\in\Cs{\P,j}\vert_k\,,\ g_+(\nu)=g}_{\st}\\
&=\bigcup_{k\in A}\bigcup_{g\in\Trans{}{j}{k}}\defset{\gamma_{g\act\eta}(+\infty)}{\eta\in\Cs{\P,k}}_{\st}\\
&=\bigcup_{k\in A}\bigcup_{g\in\Trans{}{j}{k}}g\act\defset{\gamma_\eta(+\infty)}{\eta\in\Cs{\P,k}}_{\st}\\
&=\bigcup_{k\in A}\bigcup_{g\in\Trans{}{j}{k}}g\act I_{\P,k,\st}\,,
\end{align*}
and the union is disjoint since those in~\eqref{EQN:CPCPk} and~\eqref{EQN:CPdecomp} have been.
Hence, we obtain the second relation in~(\ref{BP:intervaldecomp}\ref{BP:intervaldecompGdecomp}).
Combining it with~\eqref{EQN:Isubset} also yields the first.
The definitions of the indices and transformations involved immediately imply~(\ref{BP:intervaldecomp}\ref{BP:intervaldecompGgeod}).
And for~(\ref{BP:intervaldecomp}\ref{BP:intervaldecompback}) we argue analogously by utilizing~$t_-(\nu)$ from~\eqref{ITEM:cex:nextinter}.
This completes the proof.

\subsection*{Proof of Proposition~\ref{PROP:finram}}

Proposition~\ref{PROP:finram} is essentially a corollary of the combination of the statements~\eqref{ITEM:cex:tessellation}--\eqref{ITEM:cex:nextinter}.

Consider the tessellation of~$\H$ by~$\mathscr{B}$. Statement~(\ref{ITEM:cex:tessellation}\ref{ITEM:tessellation:structure}) implies that
\[
\bigcup_{B\in\mathscr{B}}\partial B=\Gamma\act\base{\BrU_{\P}}\,.
\]
In particular, it follows that~$B^{\circ}\cap\Gamma\act\base{\BrU_{\P}}=\varnothing$ for all~$B\in\mathscr{B}$, where~$B^{\circ}$ denotes the inner points of~$B$.
Let~$j\in A$ and~$\nu\in\Cs{\P,j,\st}$. Further let~$t_+(\nu),\,k_+(\nu),\,g_+(\nu)$ be as in~\eqref{ITEM:cex:nextinter} resp. as in the proof of Theorem~\ref{THM:cuspexpSoB}. 
Then we find~$B\in\mathscr{B}$ such that~$\base{\Cs{\P,j}}\subseteq\partial B$ and~$B^{\circ}\subseteq\Plusspp{j}$.
From the above and~\eqref{ITEM:cex:uniqueness} we infer that the number~$t_+(\nu)>0$ is such that~$\gamma_\nu(t_+(\nu))\in\partial B$.
Since, by~(\ref{ITEM:cex:tessellation}\ref{ITEM:tessellation:structure}) and~\eqref{ITEM:cex:uniqueness}, every side of~$B$ is associated with exactly two branch copies~$g_1\act\Cs{\P,k_1},\,g_2\act\Cs{\P,k_2}$, and~$B$ is finitely sided, there are only finitely many possibilities for the pair~$(k_+(\nu),g_+(\nu))$.
This yields the assertion.

\subsection*{Proof of Proposition~\ref{PROP:kundisfund}}

The notions of geometrical finiteness, exactness, and convex polyhedra follow~\cite{Ratcliffe} (see also~\cite[Definition~6.1.31]{Pohl_diss}).

Because of~\eqref{ITEM:iso:boundradii} and~\eqref{ITEM:iso:locfin}, ~\cite[Corollary~3.20]{Pohl_isofunddom} shows that~$\kund$ is a fundamental region for~$\Gamma$.
Because of~\eqref{ITEM:iso:boundradii} and the convexity of the sets~$\overline{\mathrm{int\,}\I},~\I\in\Iso{\Gamma}$, it is also connected, thus a fundamental domain for~$\Gamma$.
Said convexity further implies convexity of~$\kund$ as well.
For~$g\in\Gamma_{\RELL}$ denote by~$b_g$ the geodesic segment from~\eqref{ITEM:iso:relpart}.
Because of~\eqref{ITEM:iso:biject} the mapping~$\Gamma_{\RELL}\ni g\mapsto b_g$ is a bijection.
From the definitions of~$\REL{\Gamma}$ and~$\kund$ we infer that
\begin{align}\label{EQN:partkund}
\partial\kund=\bigcup_{g\in\Gamma_{\RELL}}\left(b_g\cap\H\right)\,,
\end{align}
where the union is essentially disjoint, meaning that for~$g,h\in\Gamma_{\RELL},~g\ne h$, the segments~$b_g,b_h$ intersect each other in at most one point.
Hence,~\eqref{ITEM:iso:locfin} and~\eqref{ITEM:iso:relpart} induce an exact side-pairing on~$\kund$.
From~\eqref{EQN:partkund} we further infer that the set of sides of~$\kund$ in~$\overline{\H}^{\geo}$ is given by~$\defset{b_g\cap\H}{g\in\Gamma_{\RELL}}$.
Because of~\eqref{ITEM:iso:locfin} this is a locally finite set.
Hence,~$\kund$ is a convex fundamental polyhedron for~$\Gamma$.
Geometrical finiteness of~$\kund$ now follows from geometrical finiteness of~$\Gamma$ and \cite[Theorem~12.4.5]{Ratcliffe}.


\bibliographystyle{amsplain}
\bibliography{pw_TObib}
\end{document}